\documentclass[sn-mathphys-num]{sn-jnl} 

\usepackage{amsmath}
\usepackage{amsthm}
\usepackage{amssymb}
\usepackage{color}
\usepackage{url}
\usepackage{graphicx}
\usepackage{ascmac}
\usepackage{float}
\usepackage{comment}
\usepackage{ifthen}
\usepackage{mathdots}
\usepackage{algorithm}
\usepackage{algorithmic}
\usepackage{mathrsfs}

\theoremstyle{definition}
\newtheorem{thm}{Theorem}

\newtheorem{lem}[thm]{Lemma}

\newtheorem{rem}[thm]{Remark}

\allowdisplaybreaks[4]

\numberwithin{equation}{section}
\numberwithin{thm}{section}

\begin{document}

\title{Computation of the exponential function of matrices by a formula without oscillatory integrals on infinite intervals}

\author[1]{\fnm{Masato} \sur{Suzuki}}\email{suzuki-masato602@g.ecc.u-tokyo.ac.jp}
\author*[2]{\fnm{Ken'ichiro} \sur{Tanaka}}\email{kenichiro@comp.isct.ac.jp}

\affil[1]{\orgdiv{Department of Mathematical Informatics}, \orgname{Graduate School of Information Science and Technology, University of Tokyo}, 
\orgaddress{\street{7-3-1 Hongo, Bunkyo-ku}, \city{Tokyo}, \postcode{113-8656}, \country{Japan}}}
\affil*[2]{\orgdiv{Department of Mathematical and Computing Science}, \orgname{School of Computing, Institute of Science Tokyo}, 
\orgaddress{\street{2-12-1 Ookayama, Meguro-ku}, \city{Tokyo}, \postcode{152-8550}, \country{Japan}}}

\abstract{
We propose a quadrature-based formula for computing the exponential function of matrices 
with a non-oscillatory integral on an infinite interval and an oscillatory integral on a finite interval. 
In the literature, 
existing quadrature-based formulas are based on the inverse Laplace transform or the Fourier transform. 
We show these expressions are essentially equivalent in terms of complex integrals 
and choose the former as a starting point to reduce computational cost.
By choosing a simple integral path, we derive an integral expression mentioned above. 
Then, we can easily apply the double-exponential formula and the Gauss-Legendre formula, 
which have rigorous error bounds.
As numerical experiments show, 
the proposed formula outperforms the existing formulas
when the imaginary parts of the eigenvalues of matrices have large absolute values. }



\maketitle


\section{Introduction}
\label{sec:intro}

This paper is concerned with numerical computation of 
the exponential function of a matrix $A \in \mathbb{C}^{m\times m}$ defined by 
\begin{align}
\exp(A) = \sum_{k=0}^{\infty} \frac{1}{k!} A^{k}. 
\label{eq:expA}
\end{align}
It appears in various problems of scientific computing 
such as
numerical solution of ordinary differential equations with the exponential integrators \cite{hochbruck2010exponential}, 
control theory \cite{franklin1998digital}, 
Markov chain process \cite{SIDJE1999345}, and
analysis of directed networks \cite{de2019analysis}. 
In response to this situation, 
many methods have been proposed for computing 
$\exp(A)$ or $\exp(A) b$ with $b \in \mathbb{C}^{m}$. 
Major examples are methods with Schur decomposition, P\'ade approximation, Newton methods, etc. 
See Higham's celebrated book \cite{higham2014functions}
and references therein for their details. 
These methods are useful for matrices with middle size. 
However, 
they are not necessarily effective for large matrices and it is difficult to parallelize them. 

Besides the methods above, 
quadrature-based methods have been studied 
for computing exponential \cite{schmelzer2006evaluating, tatsuoka2024computing} and other functions \cite{hale2008computing, tatsuoka2020algorithms, tatsuoka2020computing, tatsuoka2024preconditioning} of matrices. 
These methods
can be easily applied to large matrices, 
in particular, 
when they are either sparse or structured. 
Furthermore, 
we can easily make these methods parallel 
because their integrands can be computed independently on each node for quadrature. 
Readers can see more details in \cite[Sect. 18]{trefethen2014exponentially}.

A key for getting quadrature-based methods is 
to find integral formulas representing target functions. 
For the exponential function, 
there are two types of such integral formulas: 
the formula for the inverse Laplace transform (i.e., the Bromwich integral) \cite{cohen2007numerical}
\begin{align}
\exp(z) 
= 
\frac{1}{2 \pi \mathrm{i}} \int_{c - \mathrm{i} \infty}^{c + \mathrm{i} \infty}
\frac{\mathrm{e}^{\zeta}}{\zeta - z} \, \mathrm{d} \zeta,
\label{eq:exp_inv_L_formula} 
\end{align}
and the formulas for the Fourier transform 
\begin{align}
\exp(z |t|)
& =
\frac{1}{2 \pi} \int_{-\infty}^{\infty} \frac{2z}{x^{2} + z^{2}} \, \mathrm{e}^{\mathrm{i} t x} \, \mathrm{d}x 
\qquad \text{or}
\label{eq:exp_orig_formula_1} \\
\mathrm{i} \, \mathrm{sign}(t) \exp(z |t|)
& =
\frac{1}{2 \pi} \int_{-\infty}^{\infty} \frac{2x}{x^{2} + z^{2}} \, \mathrm{e}^{\mathrm{i} t x} \, \mathrm{d}x
\label{eq:exp_orig_formula_2}
\end{align}
for $z$ with $\mathop{\mathrm{Re}} z < 0$. 
The authors of 
\cite{talbot1979accurate,rizzardi1995modification,weideman2006optimizing,trefethen2006talbot, weideman2007parabolic}
use formulas given by deforming the integral paths of Formula~\eqref{eq:exp_inv_L_formula}. 
The deformed paths $\zeta = \zeta(s)$ are designed such that 
\(
\displaystyle
\lim_{s \to \pm \infty}
\mathop{\mathrm{Re}} \zeta(s) \to -\infty
\).  
This property realizes fast decay of the factor $\mathrm{e}^{\zeta(s)}$ in the integrand 
and contributes to fast convergence of quadrature applied to the formulas. 
Tatsuoka et al. \cite{tatsuoka2024computing}
use a formula derived from Formula~\eqref{eq:exp_orig_formula_2}
by letting $t = 1$. 
They apply the double exponential (DE) formula specialized to Fourier integrals 
\cite{ooura1999robust}
to the formula and realize accurate computation of the exponential function. 
These methods require numerical computation of oscillatory integrals on infinite intervals. 
Computing such integrals with rigorous error estimate is difficult
although there are several methods for such integrals \cite{longman1956note}.

In this paper, 
we propose a quadrature-based method to compute the exponential function of matrices 
without using oscillatory integrals on infinite intervals.
To this end, 
we aim to modify Formulas~\eqref{eq:exp_inv_L_formula}, \eqref{eq:exp_orig_formula_1}, or~\eqref{eq:exp_orig_formula_2}. 
In fact, 
they do not have essential difference
in view of complex contour integrals. 
To show this fact, 
we consider a contour $\varGamma$ surrounding $- \mathrm{i} z$ in the upper half plane. 
Then, by the Cauchy integral formula, we have
\begin{align}
\frac{1}{\pi} 
\oint_{\varGamma} \frac{\xi}{\xi^{2} + z^{2}} \mathrm{e}^{\mathrm{i} \xi} \, \mathrm{d}\xi
& = 
\frac{1}{2 \pi} 
\oint_{\varGamma} 
\left( 
\frac{1}{\xi + \mathrm{i} z} + \frac{1}{\xi - \mathrm{i} z} 
\right)
\mathrm{e}^{\mathrm{i} \xi}
\, \mathrm{d}\xi
\notag \\
& = 
\frac{1}{2 \pi} 
\oint_{\varGamma} 
\frac{1}{\xi + \mathrm{i} z} 
\mathrm{e}^{\mathrm{i} \xi}
\, \mathrm{d}\xi
= 
\frac{1}{2 \pi \mathrm{i}} 
\oint_{\varGamma} 
\frac{1}{ - (\mathrm{i} \xi - z)} 
\mathrm{e}^{\mathrm{i} \xi}
\, \mathrm{d}\xi
\notag \\
& = 
\frac{1}{2 \pi} 
\oint_{\varGamma'} 
\frac{1}{\zeta - z} 
\mathrm{e}^{\zeta}
\, \mathrm{d}\zeta,
\end{align}
where
$\Gamma'$ is a contour surrounding $z$ in the left half plane. Therefore \eqref{eq:exp_inv_L_formula} and \eqref{eq:exp_orig_formula_2} with $t=1$ is essentially equivalent. The equivalence of \eqref{eq:exp_inv_L_formula} and \eqref{eq:exp_orig_formula_1} with $t=1$ can be shown by a similar argument. Thus we start from the integral form \eqref{eq:exp_inv_L_formula} and construct a numerically efficient quadrature for the matrix exponential.

By using a special integral path, 
we propose a formula with a non-oscillatory integral on an infinite interval 
and oscillatory integral on a \emph{finite} interval. 
We apply the double exponential (DE) formula to the former integral and the Gauss-Legendre to the latter. 

The rest of this paper is organized as follows. 
In Section \ref{sec:non_osc_int_rep_of_exp}, 
we obtain an integral formula of the \emph{scalar} exponential, 
which avoids oscillatory integral on infinite intervals. 
In Section \ref{sec:num_int_for_exp}, we apply quadrature methods to the integral representation and derive rigorous error bounds of the methods. 
In Section \ref{sec:num_int_for_mat_exp}, we explain the corresponding method to compute the \emph{matrix} exponential,
and show that the error bounds can be obtained in the same manner as the scalar case.
We show numerical results in Section \ref{sec:num_exp}. In Section \ref{sec:proofs}, we show detailed proofs of the theorems presented in Sections \ref{sec:num_int_for_exp} and \ref{sec:num_int_for_mat_exp}. 
We conclude this paper in Section \ref{sec:conclusion}.

\section{Integral formula for the exponential function}
\label{sec:non_osc_int_rep_of_exp}

\subsection{Formula with a complex contour integral}

We derive a formula for $\exp(z)$ for $z \in \mathbb{C}$ with $\mathop{\mathrm{Re}} z < 0$. 
To this end, 
we consider the complex integral 
\begin{align}
\frac{1}{2 \pi \mathrm{i}} \oint_{\varGamma} \frac{\mathrm{e}^{\zeta}}{\zeta -z} \, \mathrm{d}\zeta, 
\label{eq:starting_integral}
\end{align}
where $\varGamma$ is a contour with the counterclockwise orientation. 
To give such a contour surrounding the pole $\zeta =  z$ in the left half plane, 
we define line segments $H_{\beta}(s)$ and $V_{\beta}(s,t)$ for $s,t \in \mathbb{R}$ and $\beta > 0$ as follows: 
\begin{align}
& V_{\beta}(s) := \{ z = x + \mathrm{i} y \in \mathbb{C} \mid x = s, y \in [-\beta,\beta] \},
\notag \\
& H_{\beta}(s, t) := \{ z = x + \mathrm{i} y \in \mathbb{C} \mid x \in [s, t], \ y = \beta \}.
\notag 
\end{align}
Let $\alpha$ and $r$ be real numbers with 
\begin{align}
|\mathop{\mathrm{Im}} z| < \alpha 
\quad \text{and} \quad
|\mathop{\mathrm{Re}} z| < r, 
\label{eq:cond_for_alpha_and_r}
\end{align}
respectively. 
Then, we define a contour $\varGamma_{\alpha, r}$ by 
\begin{align}
\varGamma_{\alpha, r}
:= 
V_{\alpha}(0) \cup H_{\alpha}(-r,0)\cup V_{\alpha}(-r) \cup H_{-\alpha}(-r,0)
\notag
\end{align}
with the counterclockwise orientation.
By definition, 
the contour $\varGamma_{\alpha, r}$ surrounds the pole $\zeta = z$. 
Therefore 
it follows from the residue theorem that 
\begin{align}
\frac{1}{2\pi \mathrm{i}} \oint_{\varGamma_{\alpha, r}} \frac{\mathrm{e}^{\zeta}}{\zeta - z}  \, \mathrm{d}\zeta
& = 
\frac{1}{2 \pi \mathrm{i}} \cdot 2 \pi \mathrm{i} \mathop{\mathrm{Res}}_{\zeta =  z}
\left(
\frac{\mathrm{e}^{\zeta}}{\zeta - z}
\right)
= 
\mathrm{e}^{z}.
\label{eq:cont_int_formula_of_exp}
\end{align}

\subsection{Derivation of a formula by taking a limit}

Formula~\eqref{eq:cont_int_formula_of_exp} is rewritten in the form
\begin{align}
\exp(z) 
= & \, 
\frac{1}{2 \pi \mathrm{i}} \left(
\int_{r}^{0} \frac{\mathrm{e}^{\mathrm{i} \alpha - x}}{\mathrm{i}\alpha -x -z}\, \mathrm{d}x
+ \int_{0}^{r} \frac{\mathrm{e}^{-\mathrm{i}\alpha-x}}{-\mathrm{i}\alpha -x-z} \, \mathrm{d}x
\right.
\notag \\
& \, 
\phantom{\frac{1}{\pi \mathrm{i}} (}
\left.
+ \int_{-\alpha}^{\alpha} 
\frac{\mathrm{i}\mathrm{e}^{\mathrm{i}x}}{\mathrm{i} x-z} \, \mathrm{d}x
+ 
\int_{\alpha}^{-\alpha} \frac{\mathrm{i}\mathrm{e}^{\mathrm{i}x-r}}{\mathrm{i}x-r-z} \,  \mathrm{d}x
\right)
\notag \\
= & \, 
\frac{1}{2\pi \mathrm{i}} \int_{0}^{r} 
\left(
\frac{\mathrm{e}^{\mathrm{i} \alpha}}{z-\mathrm{i}\alpha+x} -
\frac{\mathrm{e}^{-\mathrm{i}\alpha}}{z+\mathrm{i}\alpha +x} 
\right)
\mathrm{e}^{- x} \, \mathrm{d}x
\notag \\
& \, 
+ \frac{\alpha}{2\pi} \int_{-1}^{1} 
\frac{\mathrm{e}^{\mathrm{i}\alpha x}}{\mathrm{i}\alpha x-z} \, \mathrm{d}x
- \bar{I}_{r}, 
\label{eq:exp_rewritten_form}
\end{align}
where 
\begin{align}
\bar{I}_{r} 
:= 
\frac{1}{2\pi }
\int_{-\alpha}^{\alpha} \frac{\mathrm{e}^{\mathrm{i}x-r}}{z-\mathrm{i}x +r}  \, \mathrm{d}x. 
\notag
\end{align}
Below we show that
\(
\displaystyle
\lim_{r \to \infty} \bar{I}_{r} = 0
\). 
Indeed, for $r$ with $r > |z|$, we have
\begin{align}
|\bar{I}_{r}| 
& \leq
\frac{1}{2\pi}
\int_{-\alpha}^{\alpha} \frac{\mathrm{e}^{-r}}{|z-\mathrm{i}x+r|} \,  \mathrm{e}^{-r} \, \mathrm{d}x
\leq
\frac{1}{2\pi}
\int_{-\alpha}^{\alpha} \frac{ \mathrm{e}^{-r}}{|r-\mathrm{i}x|- |z|} \,  \mathrm{d}x
\notag \\
& \leq 
\frac{1}{2\pi}
\int_{-\alpha}^{\alpha} \frac{\mathrm{e}^{-r}}{r-|z|} \, \mathrm{d}x
=
\frac{\alpha\mathrm{e}^{-r}}{\pi (r-|z|)}\, . 
\notag 
\end{align}
Because the value in the right-hand side goes to $0$ as $r \to \infty$, we have 
\(
\displaystyle
\lim_{r \to \infty} I_{r} = 0
\). 
Therefore, by letting $r \to \infty$ in \eqref{eq:exp_rewritten_form}, 
we get the following theorem giving a formula for $\exp(z)$. 

\begin{thm}
\label{thm:int_formula_for_exp}
Let $z$ be a complex number with $\mathop{\mathrm{Re}} z < 0$. 
Then, for any real number $\alpha$ with $|\mathop{\mathrm{Im}} z| < \alpha$, 
the following equality holds: 
\begin{align}
\exp(z)
=
\frac{1}{2\pi\mathrm{i}} \int_{0}^{\infty} 
\left(
\frac{\mathrm{e}^{\mathrm{i} \alpha}}{z-\mathrm{i}\alpha+x} -
\frac{\mathrm{e}^{- \mathrm{i} \alpha} }{z+\mathrm{i}\alpha +x}
\right)
\mathrm{e}^{- x} \, \mathrm{d}x
+ 
\frac{\alpha}{2\pi} \int_{-1}^{1} 
\frac{\mathrm{e}^{\mathrm{i}\alpha x}}{\mathrm{i}\alpha x-z} \, \mathrm{d}x.
\label{eq:derived_formula_for_exp}
\end{align}
In the special case $z \in \mathbb{R}$, this formula is simplified as follows: 
\begin{align}
\exp(z)
=
\frac{1}{\pi} 
\mathop{\mathrm{Im}} \left(
\int_{0}^{\infty} 
\frac{\mathrm{e}^{\mathrm{i}\alpha}}{z-\mathrm{i}\alpha+x} \, 
\mathrm{e}^{- x} \, \mathrm{d}x
\right)
+ 
\frac{\alpha}{2\pi} \int_{-1}^{1} 
\frac{\mathrm{e}^{\mathrm{i}\alpha x}}{\mathrm{i}\alpha x-z} \, \mathrm{d}x.
\label{eq:derived_formula_for_exp_real_z}
\end{align}
\end{thm}

For later use, we define 
$I_{\alpha}(z)$ and $J_{\alpha}(z)$ by 
the first and second integral of Formula~\eqref{eq:derived_formula_for_exp}, respectively, i.e., 
\begin{align}
& I_{\alpha}(z)
:=
\frac{1}{2\pi\mathrm{i}} \int_{0}^{\infty} 
\left(
\frac{\mathrm{e}^{\mathrm{i} \alpha}}{z-\mathrm{i}\alpha+x} -
\frac{\mathrm{e}^{- \mathrm{i} \alpha} }{z+\mathrm{i}\alpha +x}
\right)
\mathrm{e}^{- x} \, \mathrm{d}x
\quad \text{and}
\label{eq:I_non_osc_int} \\
& J_{\alpha}(z)
:= 
\frac{\alpha}{2\pi} \int_{-1}^{1} 
\frac{\mathrm{e}^{\mathrm{i}\alpha x}}{\mathrm{i}\alpha x-z} \, \mathrm{d}x.
\label{eq:J_osc_int} 
\end{align}

\begin{rem}
Our idea is based on that considered by Milovanovi{\'c} (2017) \cite{milovanovic2017computing}. 
They use a rectangular contour for computing Fourier-type integrals 
to avoid oscillatory integrals on infinite intervals. 
\end{rem}

\section{Approximation formula for the exponential function by numerical integration}
\label{sec:num_int_for_exp}

To compute $\exp(z)$ by Formula~\eqref{eq:derived_formula_for_exp}, 
we apply numerical integration formulas to 
$I_{\alpha}(z)$ in \eqref{eq:I_non_osc_int} 
and 
$J_{\alpha}(z)$ in \eqref{eq:J_osc_int}. 

\subsection{Approximation formula for $I_{\alpha}(z)$}

Let $f_{\alpha}(z, x)$ be defined by
\begin{align}
f_{\alpha}(z, x) 
:=
\frac{1}{2\pi\mathrm{i}} 
\left(
\frac{\mathrm{e}^{\mathrm{i} \alpha}}{z-\mathrm{i}\alpha+x} -
\frac{\mathrm{e}^{- \mathrm{i} \alpha} }{z+\mathrm{i}\alpha +x}
\right)
\mathrm{e}^{- x}.
\label{eq:def_integrand_of_non_osc}
\end{align}
Then 
the function $f_{\alpha}(z, \cdot)$ is a non-oscillatory function 
with exponential decay on the semi-infinite interval $[0, \infty)$
and 
$I_{\alpha}(z)$ is given by an integral of $f_{\alpha}(z, \cdot)$ on $[0, \infty)$.  
In addition, 
we can regard $f_{\alpha}(z, \cdot)$ as an analytic function in a complex domain. 
More precisely, 
if $z$ is given by $z = u + \mathrm{i} v$ for $u, v \in \mathbb{R}$ with $u < 0$, 
the function $f_{\alpha}(z, \cdot)$ is an analytic function in
\begin{align}
\{ x \in \mathbb{C} \mid x \neq - u - \mathrm{i} (v \pm \alpha) \}. 
\label{eq:domain_of_f_alpha}
\end{align}

For computing integrals on $[0,\infty)$ of such analytic functions with exponential decay, 
the double-exponential (DE) formula with variable transformation
\begin{align}
\phi(t) 
:= 
\log \left(
1 + \exp(\pi \sinh t)
\right)
\label{eq:def_DE_trans_4}
\end{align}
is very useful 
\cite{
takahasi1974double, 
tanaoka2023NMVTeng, 
tanaka2009function}. 
Therefore we apply the DE formula to $I_{\alpha}(z)$ and get a numerical integration formula
\begin{align}
\tilde{I}_{\alpha, n, h}(z)
:= 
h \sum_{k=-n}^{n} f_{\alpha}(z, \phi(kh)) \, \phi'(kh), 
\label{eq:DE_for_I_alpha}
\end{align}
where $n$ is a positive integer setting the number of nodes and $h > 0$ is a grid spacing. 

\begin{rem}
We can also use the Gauss-Laguerre quadrature when computing integrals of exponentially convergent integrands like \eqref{eq:def_integrand_of_non_osc} on $[0,\infty)$. 
As shown in Section \ref{sec:numerical_experiment_DE_GL}, it outperforms the DE formula in some cases, and can also be a practical choice. However we choose the DE formula, due to the superiorities mentioned in Section \ref{sec:numerical_experiment_DE_GL}.
\end{rem}

By determining an appropriate $h$ for a given integer $n$ in Formula~\eqref{eq:DE_for_I_alpha}, 
we can bound the error $|I_{\alpha}(z) - \tilde{I}_{\alpha, n, h}(z)|$ in terms of $n$
as shown in the following theorem. 
We prove this theorem in Section~\ref{sec:proof_total_error_of_I_alpha_n_h}.

\begin{thm}
\label{thm:total_error_of_I_alpha_n_h}
Let $z$ be a complex number with $\mathop{\mathrm{Re}} z < 0$ and 
let $\alpha$ be a real number with $|\mathop{\mathrm{Im}} z| + 2\pi < \alpha$. 
Furthermore, 
let $d$ be a real number with
\begin{align}
0 < d < \mathrm{arctan} \left( \frac{\alpha - |\mathop{\mathrm{Im}} z| - 2 \pi}{- \mathop{\mathrm{Re}} z + \log 2} \right). 
\label{eq:appropriate_d_for_D_d}
\end{align}
For a positive integer $n$ with $n > 1/(4d)$, 
we define $h$ by
\begin{align}
h := \frac{\log (4dn)}{n}. 
\label{eq:def_of_h}
\end{align}
Then the error of Formula~\eqref{eq:DE_for_I_alpha} is bounded as follows:
\begin{align}
|I_{\alpha}(z) - \tilde{I}_{\alpha, n, h}(z)|
\leq 
c_{d} K_{z, \alpha, d} 
\exp \left(
- \frac{2\pi d n}{\log(4dn)}
\right),
\label{eq:total_error_of_I_alpha_n_h}
\end{align}
where 
$c_{d}$ is a positive real number depending only on $d$ and 
$K_{z, \alpha, d}$ is defined by 
\begin{align}
K_{z, \alpha, d} 
:= 
\frac{1/\pi}{(\alpha - |\mathop{\mathrm{Im}} z| - 2 \pi) \cos d - (-\mathop{\mathrm{Re}} z + \log 2) \sin d}.
\label{eq:def_of_K_alpha_z}
\end{align}
\end{thm}

\subsection{Approximation formula for $J_{\alpha}(z)$}
Let $g_\alpha(z,x)$ be defined by
\begin{align}
g_{\alpha}(z, x) 
:=
\frac{\alpha}{2\pi} 
\frac{\mathrm{e}^{\mathrm{i} \alpha x}}{\mathrm{i}\alpha x-z}.
\label{eq:def_hat_g_alpha_z}
\end{align}
The function $g_{\alpha}(z, \cdot)$ is analytic in a complex region containing the interval $[-1, 1]$
and $J_{\alpha}(z)$ is the integral of $g_{\alpha}(z, \cdot)$ with weight $1$. 
Therefore 
we apply the Gauss--Legendre quadrature formula to $J_{\alpha}(z)$ and get a numerical integration formula
\begin{align}
\tilde{J}_{\alpha, N}(z)
:= 
\sum_{i=1}^{N} w_{i} \, g_{\alpha}(z, t_{i}), 
\label{eq:def_tilde_J_alpha_N}
\end{align}
where $t_{i}$ are the zeros of the Legendre polynomial of degree $N$ 
and $w_{i}$ are the corresponding weights. 
To bound the error $|J_{\alpha}(z) - \tilde{J}_{\alpha, N}(z)|$, 
by letting $\delta$ be a real number with
$0 < \delta < |\mathop{\mathrm{Re}} z|/\alpha$, 
we define 
\begin{align}
\rho_{z, \alpha, \delta} 
:= 
|\mathop{\mathrm{Re}} z|/\alpha - \delta + \sqrt{(|\mathop{\mathrm{Re}} z|/\alpha - \delta)^{2} + 1}.
\label{eq:def_rho_z_alpha}
\end{align}
Then the error is bounded as shown in the following theorem, 
which we prove in Section~\ref{sec:proof_total_error_of_J_alpha_N}. 

\begin{thm}
\label{thm:total_error_of_J_alpha_N}
Let $z$ be a complex number with $\mathop{\mathrm{Re}} z < 0$ and 
let $\alpha$ be a real number with $|\mathop{\mathrm{Im}} z| < \alpha$. 
In addition, 
let $\delta$ be a real number with $0 < \delta < |\mathop{\mathrm{Re}} z|/\alpha$ and 
define $\rho_{z,\alpha, \delta}$ by~\eqref{eq:def_rho_z_alpha}.
Then, for $N$ with $N \geq 2$, 
the error of Formula~\eqref{eq:def_tilde_J_alpha_N} is bounded as follows: 
\begin{align}
|J_{\alpha}(z) - \tilde{J}_{\alpha, N}(z)| 
\leq 
\frac{32 \exp(|\mathop{\mathrm{Re}} z |)}{15 \pi \delta} \cdot
\frac{\rho_{z,\alpha, \delta}^{-2(N-1)}}{\rho_{z,\alpha, \delta}^{2} - 1}. 
\label{eq:total_error_of_J_alpha_N}
\end{align}
\end{thm}

\begin{rem}
The integrand $g_\alpha(z,\cdot)$ becomes highly oscillatory when $\alpha$ is huge, 
and the convergence of the Gauss-Legendre quadrature becomes slow. 
Filon and Levin-type methods \cite{deano2017computing} can also be considered 
for computing highly oscillatory integrals on finite intervals. 
However, error analysis of these methods focuses on asymptotic analysis with respect to the frequency, and explicit error bounds with the number of sample points are not obtained. In addition, these methods typically require solving linear equations of function values. Therefore, constructing a quadrature based on these methods for a matrix-valued integrand is challenging.
\end{rem}

\subsection{Choice of $\alpha$}
\label{sec:choice_of_alpha}

Here we consider how to choose the parameter $\alpha$,  
based on Theorems~\ref{thm:total_error_of_I_alpha_n_h} and~\ref{thm:total_error_of_J_alpha_N}. 
Since we compute two quadrature formulas simultaneously, the errors of the two formulas should decrease in close order of magnitude. 
If we use $(2n+1)$ points for the DE formula and $N=kn~(k>0)$ points for the Gauss-Legendre formula, the main factors of the bounds for $\tilde{I}_{\alpha, n, h}(z)$ and $\tilde{J}_{\alpha, N}(z)$ are
\begin{align}
(\exp(- \pi d ))^{2n/\log(4dn)} 
\qquad
\text{and}
\qquad
(\rho_{z,\alpha, \delta}^{-2})^{kn}, 
\end{align}
respectively. 
Then it is sensible to choose $\alpha$ that satisfies
\begin{align}
\exp(2\pi d) = \rho_{z,\alpha, \delta}^{2k}
& \iff
\exp(\pi d/k) = \rho_{z,\alpha, \delta}
\notag \\
& \iff
\frac{\exp(\pi d/k) - \exp(-\pi d/k)}{2} = \frac{1}{\alpha} |\mathop{\mathrm{Re}} z| - \delta
\notag \\
& \iff
\sinh\left( \frac{\pi}{k} d \right) = \frac{1}{\alpha} |\mathop{\mathrm{Re}} z| - \delta, 
\label{eq:pre_criterion_choose_alpha}
\end{align}
where the second equivalence relation is owing to~\eqref{eq:def_rho_z_alpha}. 
Theorem~\ref{thm:total_error_of_I_alpha_n_h} 
indicates that $d$ with~\eqref{eq:appropriate_d_for_D_d} should be as large as possible 
to make the convergence rate of $\tilde{I}_{\alpha, n, h}(z)$ fast. 
Therefore 
we assume that $d$ is equal to the upper bound in~\eqref{eq:appropriate_d_for_D_d}
and 
neglect $\delta$ in~\eqref{eq:pre_criterion_choose_alpha} for simplicity
to get
\begin{align}
\sinh\left( \frac{\pi}{k} \, \mathrm{arctan} \left( \frac{\alpha - |\mathop{\mathrm{Im}} z| - 2 \pi}{- \mathop{\mathrm{Re}} z + \log 2} \right) \right) 
= 
\frac{1}{\alpha} |\mathop{\mathrm{Re}} z|.
\label{eq:criterion_choose_alpha}
\end{align}
On the one hand, 
the function in the left-hand side is increasing for $\alpha > |\mathop{\mathrm{Im}} z| + 2 \pi$ and converges to $\sinh(\pi^2/2k)$ as $\alpha \to \infty$.  
On the other hand, 
the function in the right-hand side is decreasing for $\alpha > 0$ and zero as $\alpha \to \infty$.
Therefore for any $k$,
Equation~\eqref{eq:criterion_choose_alpha}
has a unique solution $\alpha_{k}$ with $\alpha_{k} > |\mathop{\mathrm{Im}} z| + 2 \pi$. 
\begin{rem}
In the matrix version algorithm that we show in Section \ref{sec:num_int_for_mat_exp}, we compute $4n+2$ resolvents for the DE formula and $N$ resolvents for the Gauss-Legendre quadrature (see \eqref{eq:f_alpha_z_eq_A}, \eqref{eq:hat_g_alpha_z_eq_A}, and \eqref{eq:approx_exp_A}). 
Therefore, $k=4$ is a natural choice to equalize the computational complexity for the two quadratures. The best choice of $k$ with respect to the numerical performance is discussed in Section \ref{sec:numerical_experiment_alpha_k}. 
\end{rem}

\section{Approximation formula for the exponential function of matrices}
\label{sec:num_int_for_mat_exp}

Let
$m$ 
be a positive integer and 
let $A \in \mathbb{C}^{m \times m}$
be an $m \times m$ matrix with eigenvalues $\lambda_{1}, \ldots, \lambda_{m}$. 
We assume that $\mathop{\mathrm{Re}} \lambda_{i} < 0$ holds for $i = 1,\ldots, m$. 
To derive a formula for $\exp(A)$, 
we substitute $z = A$ into $I_{\alpha}(z)$ and $J_{\alpha}(z)$. 

Then, 
the function $\exp(A)$ is given by the sum of
\begin{align}
& I_{\alpha}(A)
=
\int_{0}^{\infty} f_{\alpha}(A, x) \, \mathrm{d}x
\quad 
\text{and}
\notag \\
& J_{\alpha}(A)
=
\int_{-1}^{1} g_{\alpha}(A, x) \, \mathrm{d}x, 
\notag
\end{align}
where
\begin{align}
f_{\alpha}(A, x)
= & \, 
\frac{\mathrm{e}^{- x}}{2\pi \mathrm{i}} 
\Big(
\mathrm{e}^{\mathrm{i} \alpha}
\left( (x-\mathrm{i} \alpha)I + A \right)^{-1} 
\notag \\
& \phantom{\frac{\mathrm{e}^{- y}}{\pi} \Big(} \, 
- \, 
\mathrm{e}^{- \mathrm{i} \alpha}
\left( (x + \mathrm{i} \alpha) I + A \right)^{-1}  
\Big)
\qquad
\text{and} 
\label{eq:f_alpha_z_eq_A} \\
g_{\alpha}(A, x)
= & \, 
\frac{\alpha}{2\pi} \exp(\mathrm{i}\alpha x) \left(\mathrm{i}\alpha x I-A \right)^{-1}. 
\label{eq:hat_g_alpha_z_eq_A} 
\end{align}
By applying the numerical integration formulas in~\eqref{eq:DE_for_I_alpha} and~\eqref{eq:def_tilde_J_alpha_N}
to $I_{\alpha}(A)$ and $J_{\alpha}(A)$, respectively, 
we get an approximation formula
\begin{align}
\exp(A)
& \approx
\tilde{I}_{\alpha, n , h}(A) + \tilde{J}_{\alpha, N}(A) 
\notag \\
& = 
h \sum_{k = -n}^{n} f_{\alpha}(A, \phi(kh)) \phi'(kh)
+ 
\sum_{i=1}^{N} w_{i} \, g_{\alpha}(A, t_{i}).
\label{eq:approx_exp_A}
\end{align}
We bound the error of this formula 
with respect to a matrix norm $\| \cdot \|$ 
like the $2$-norm $\| \cdot \|_{2}$ and 
the Frobenius norm $\| \cdot \|_{\mathrm{F}}$. 
Then, we give upper bounds for 
$\| I_{\alpha}(A) - \tilde{I}_{\alpha, n , h}(A) \|$
and 
$\| J_{\alpha}(A) - \tilde{J}_{\alpha, N}(A) \|$
in the following theorems. 

\begin{thm}
\label{thm:total_error_of_I_alpha_n_h_for_A}
Let $A \in \mathbb{C}^{m \times m}$
be an $m \times m$ matrix with eigenvalues $\lambda_{1}, \ldots, \lambda_{m}$
with $\mathop{\mathrm{Re}} \lambda_{i} < 0$ 
and
let $\alpha$ be a real number with 
\begin{align}
\max_{1 \leq i \leq m} |\mathop{\mathrm{Im}} \lambda_{i} | + 2\pi < \alpha. 
\notag
\end{align}
Furthermore, 
let $d$ be a real number with
\begin{align}
0 < d < 
\min_{1 \leq i \leq m}
\mathrm{arctan} \left( \frac{\alpha - |\mathop{\mathrm{Im}} \lambda_{i}| - 2 \pi}{- \mathop{\mathrm{Re}} \lambda_{i} + \log 2} \right). 
\label{eq:appropriate_d_for_D_d_for_A}
\end{align}
For a positive integer $n$ with $n > 1/(4d)$, 
we define $h$ by
\begin{align}
h := \frac{\log (4dn)}{n}. 
\label{eq:def_of_h_for_A}
\end{align}
Then 
the error 
$\| I_{\alpha}(A) - \tilde{I}_{\alpha, n , h}(A) \|$
is bounded as follows: 
\begin{align}
\| I_{\alpha}(A) - \tilde{I}_{\alpha, n, h}(A) \|
\leq 
c_{d} L_{A, \alpha, d, \| \cdot \|} 
\exp \left(
- \frac{2\pi d n}{\log(4dn)}
\right),
\label{eq:total_error_of_I_alpha_n_h_for_A}
\end{align}
where 
$c_{d}$ is a positive real number depending only on $d$ and 
$L_{A, \alpha, d, \| \cdot \|}$ is a positive real number 
depending only on $A$, $\alpha$, $d$, and the matrix norm $\| \cdot \|$. 
\end{thm}

\begin{thm}
\label{thm:total_error_of_J_alpha_N_for_A}
Let $A \in \mathbb{C}^{m \times m}$
be an $m \times m$ matrix with eigenvalues $\lambda_{1}, \ldots, \lambda_{m}$
with $\mathop{\mathrm{Re}} \lambda_{i} < 0$ 
and
let $\alpha$ be a real number with 
\begin{align}
\max_{1 \leq i \leq m} |\mathop{\mathrm{Im}} \lambda_{i} | < \alpha. 
\notag
\end{align}
In addition, let
\begin{align}
\eta_{A}
:=
\min_{1 \leq i \leq m} |\mathop{\mathrm{Re}} \lambda_{i}|
\label{eq:def_eta_A}
\end{align}
and let $\delta$ be a real number with $0 < \delta < (2/\alpha) \eta_{A}$ and 
define $\rho_{A, \alpha, \delta}$ by
\begin{align}
\rho_{A, \alpha, \delta}
:=
\eta_{A}/\alpha - \delta + \sqrt{(\eta_{A}/\alpha - \delta)^{2} + 1}.
\label{eq:def_rho_z_alpha_for_A}
\end{align}
Then, for $N$ with $N \geq 2$, 
the error $\| J_{\alpha}(A) - \tilde{J}_{\alpha, N}(A) \|$ is bounded as follows: 
\begin{align}
\| J_{\alpha}(A) - \tilde{J}_{\alpha, N}(A) \|
\leq 
\frac{64 \, C_{A, \alpha, \delta, \| \cdot \|}}{15} \cdot
\frac{\rho_{A,\alpha, \delta}^{-2(N-1)}}{\rho_{A,\alpha, \delta}^{2} - 1}, 
\label{eq:total_error_of_J_alpha_N_for_A}
\end{align}
where 
$C_{A, \alpha, \delta, \| \cdot \|}$ 
is a positive real number 
depending only on $A$, $\alpha$, $d$, and the matrix norm $\| \cdot \|$. 
\end{thm}

We can prove 
Theorems~\ref{thm:total_error_of_I_alpha_n_h_for_A} and~\ref{thm:total_error_of_J_alpha_N_for_A}
in a similar manner to the proofs of 
Theorems~\ref{thm:total_error_of_I_alpha_n_h} and~\ref{thm:total_error_of_J_alpha_N}, 
respectively. 
We prove these theorems in Sections~\ref{sec:proof_I_alpha_n_h_for_A} and~\ref{sec:proof_J_alpha_N_for_A}. 

\begin{rem}
The explicit forms of 
$L_{A, \alpha, d, \| \cdot \|}$ in Theorem~\ref{thm:total_error_of_I_alpha_n_h_for_A}
and 
$C_{A, \alpha, d, \| \cdot \|}$ in Theorem~\ref{thm:total_error_of_J_alpha_N_for_A}
are given 
by~\eqref{eq:def_M_A_alpha_d_norm} in Section~\ref{sec:proof_I_alpha_n_h_for_A} 
and~\eqref{eq:def_C_A_alpha_delta} in Section~\ref{sec:proof_J_alpha_N_for_A}, 
respectively. 
\end{rem}

\section{Numerical experiments}
\label{sec:num_exp}

All the tests were computed by MATLAB R2023b on a MacBook Pro with M1 8-Core CPU and 8GB RAM. 
We generated $4$ test matrices $A_1,\dots,A_4$ by the following procedure.
\begin{description}
\item[Step 1.] Generate an orthogonal matrix $Q\in \mathbb{R}^{100\times 100}$ by a QR decomposition of a random $100\times 100$ matrix.
\item[Step 2.] For each $i=1,\dots,4$, generate $100$ eigenvalues by independently sampling from a uniform distribution on the region $\Omega_i$ defined as follows:
\begin{align}
\Omega_1&:=[-100,-5],\notag\\
\Omega_2&:=\{z\in \mathbb{C}\mid \mathop{\mathrm{Re}}z\in [-100,-5],\, \mathop{\mathrm{Im}}z\in [-10,10]\},\notag\\
\Omega_3&:=\{z\in \mathbb{C}\mid \mathop{\mathrm{Re}}z\in [-100,-5],\, \mathop{\mathrm{Im}}z\in [-100,100]\},\notag\\
\Omega_4&:=\{z\in \mathbb{C}\mid \mathop{\mathrm{Re}}z\in [-100,-5],\, \mathop{\mathrm{Im}}z\in [-1000,1000]\notag\}.
\end{align}
Let $D_i \in \mathbb{C}^{n\times n}$ be the diagonal matrix with the sampled eigenvalues in the diagonal elements.

\item[Step 3.] Define $A_i\,(i=1,\dots,4)$ as follows:
\begin{align}
A_i:=QD_iQ^\top.\notag
\end{align}
\end{description}
The reference matrix $\exp(A_i)$ was calculated by
\[\exp(A_i)=Q\exp(D_i)Q^\top,\]
where $\exp(D_i)$ was calculated by applying the scalar exponential function to the diagonal elements.
\subsection{Convergence of the Proposed and Existing Methods}
\label{sec:numerical_experiment_1}
In this section, we compare the convergence of our proposed method with three existing quadrature-based methods denoted by ``Talbot'', ``Fixed Talbot'', and ``Tatsuoka DE'' explained below.

To compute the Bromwich integral \eqref{eq:exp_inv_L_formula}, several methods of reforming the contour have been proposed.
We use the Talbot contour, which is proposed in \cite{talbot1979accurate} and optimized by Weideman \cite{weideman2006optimizing}. Weideman obtains an improved rate by expanding the contour proportional to the number of evaluation points. However, this expansion leads to catastrophic cancellation when the contour becomes huge. Therefore, we also use a fixed contour of Talbot. We denote these methods by ``Talbot'' and ``Fixed Talbot''.

Tatsuoka et al. \cite{tatsuoka2024computing} proposed a method that computes a Fourier-type integral on the half-infinite interval on the real axis by using a special type of double exponential transform \cite{ooura1999robust}. We denote this method by ``Tatsuoka DE''.

For the proposed method, we used $k=4~(N=4n)$ to equalize the number of resolvent computing for the DE formula and Gauss-Legendre quadrature.
We used \eqref{eq:criterion_choose_alpha} to determine $\alpha$ by setting $k=4$, $|\mathop{\mathrm{Re}}z|=5$ and substituting $|\mathop{\mathrm{Im}}z|$ by the largest absolute imaginary part of the eigenvalues of each test matrix (0, 10, 100, 1000, respectively). 

Figure \ref{fig:exp_1} shows the numerical results. Because our proposed method requires to calculate more resolvents for a single evaluation point (for computing $I_\alpha (A)$), we compare the convergence by the total number of resolvent computation. 
When the eigenvalues are real ($A_1$), 
Talbot shows the fastest convergence. However, expanding the Talbot contour leads to cancellation errors, and the error does not converge. The proposed method converges fast in every case, especially when the imaginary part of the eigenvalues is large ($A_3, A_4$) and other methods fail.

\begin{figure}[htbp]
  \begin{minipage}[b]{0.49\linewidth}
    \centering
    \includegraphics[bb = 0 0 960 720, width=0.9\linewidth]{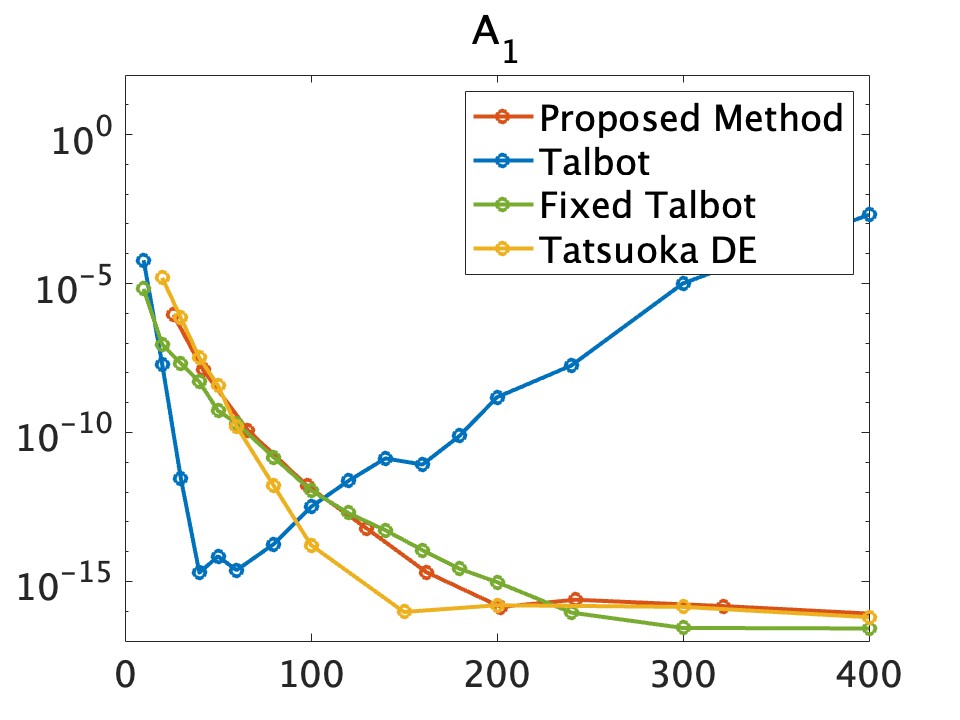}
  \end{minipage}
  \begin{minipage}[b]{0.49\linewidth}
    \centering
    \includegraphics[bb = 0 0 960 720, width=0.9\linewidth]{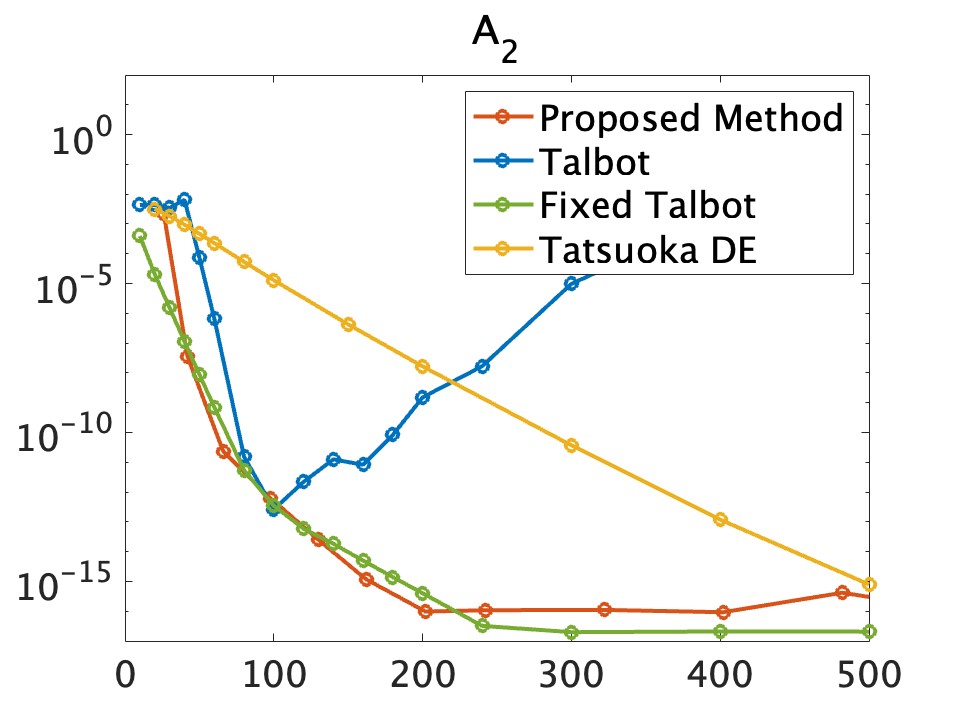}
  \end{minipage}
   \begin{minipage}[b]{0.49\linewidth}
    \centering
    \includegraphics[bb = 0 0 960 720, width=0.9\linewidth]{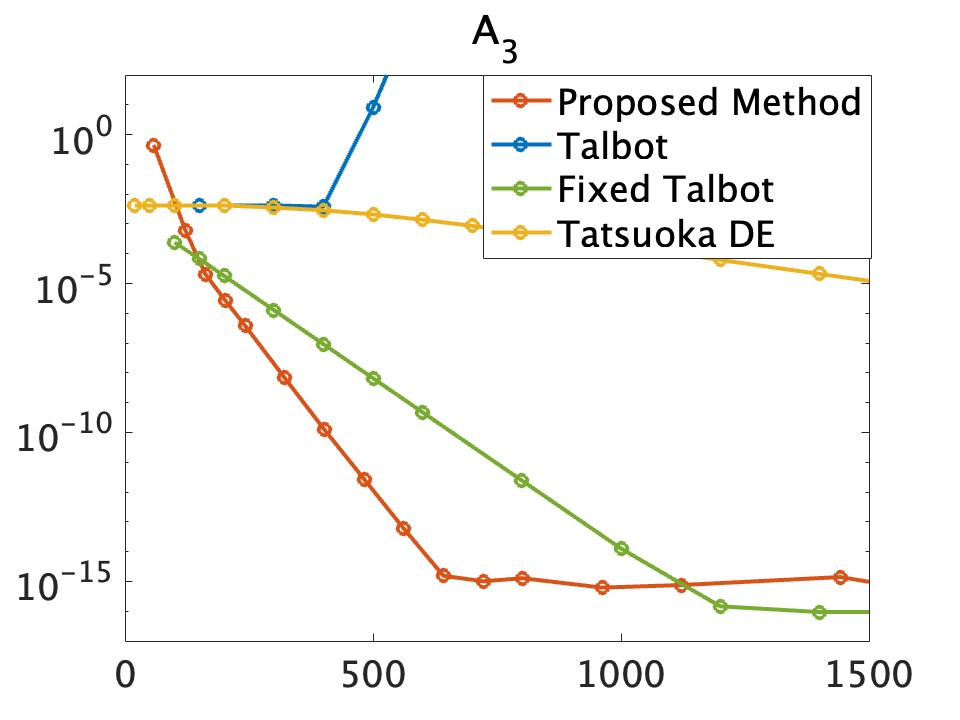}
  \end{minipage}
  \begin{minipage}[b]{0.49\linewidth}
    \centering
    \includegraphics[bb = 0 0 960 720, width=0.9\linewidth]{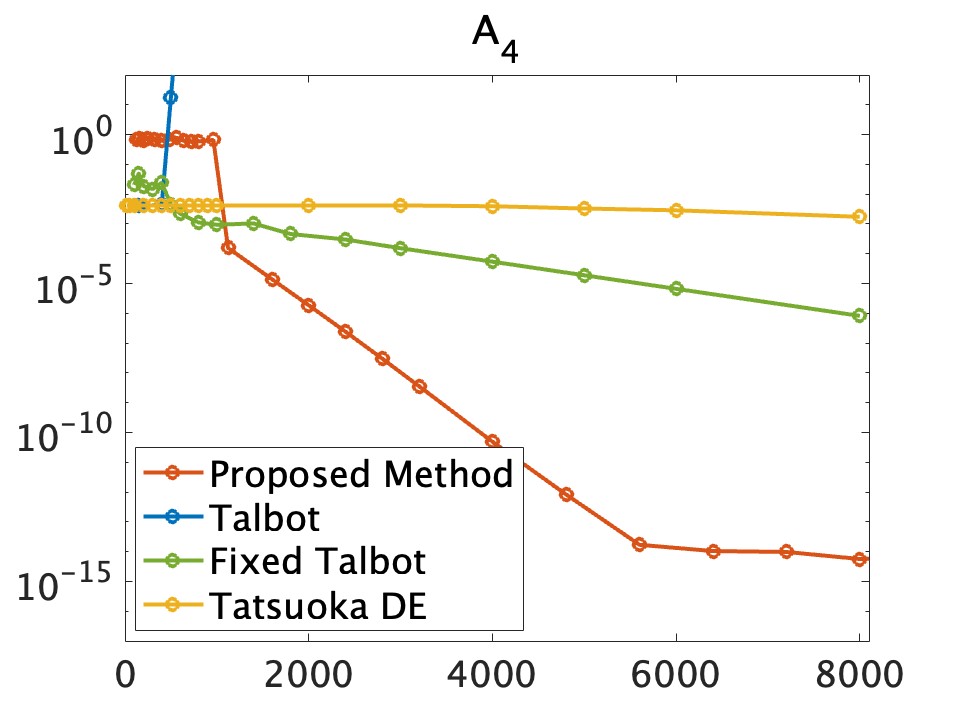}
  \end{minipage}
  \caption{The horizontal axis shows the total number of resolvent calculations, and the vertical axis shows the absolute error in the matrix 2-norm. Note that the scale of the horizontal axis is different for each figure.}
  \label{fig:exp_1}
\end{figure}

\subsection{Setting of the Parameters $\alpha$ and $k$}
\label{sec:numerical_experiment_alpha_k}
In this section, we consider the optimal choice of $\alpha$ and $k\,(=N/n)$. 
In Section \ref{sec:choice_of_alpha}, our method chooses the natural $k=4$ and solves equation \eqref{eq:criterion_choose_alpha} to obtain $\alpha$. 
Although our choice is theoretical in the sense that it roughly equalizes the order of the two errors, other choices of $\alpha,\, k$ can be considered to improve the numerical performance. 

We use $A_3$ for the test matrix and compare the numerical error of the proposed method with several pairs of $\alpha$ and $k$.
The $\alpha$'s we used for the experiment are the solution $\alpha_k$ of \eqref{eq:criterion_choose_alpha} with $z=-5+100\mathrm{i}$. Table \ref{table:value_alpha_k} shows the values of $\alpha_k$ for several $k$'s.

\begin{table}[htbp]
\caption{Values of $\alpha_k$ with $z=-5+100\mathrm{i}$}\label{table:value_alpha_k}
\begin{tabular}{cccccc} \hline
$\alpha_1$ & $\alpha_2$ & $\alpha_4$ & $\alpha_8$ & $\alpha_{16}$ &  $\alpha_{32}$\\ \hline
$106.3683$ & $106.4534$ & $106.6234$ & $106.9638$ & $107.6550$ & $109.1497$\\ \hline
\end{tabular}
\end{table}

Figure \ref{fig:exp_alpha_k} shows the numerical results. 
The parameters we used in Section \ref{sec:numerical_experiment_1} is the yellow line of the figure $k=4$ with $\alpha = \alpha_4$. 
Although this pair of parameters converge at around $600$ resolvent calculations, using larger $\alpha$'s and $k$'s, such as $k=8,16$ with $\alpha = \alpha_{16}, \alpha_{32}$ converges at around $400$ resolvent calculations and is more efficient. 
This suggests that more examination is needed to the setting of $\alpha$ and $k$ to improve the numerical performance. 

\begin{figure}[htbp]
  \begin{minipage}[b]{0.49\linewidth}
    \centering
    \includegraphics[bb = 0 0 960 720, width=0.9\linewidth]{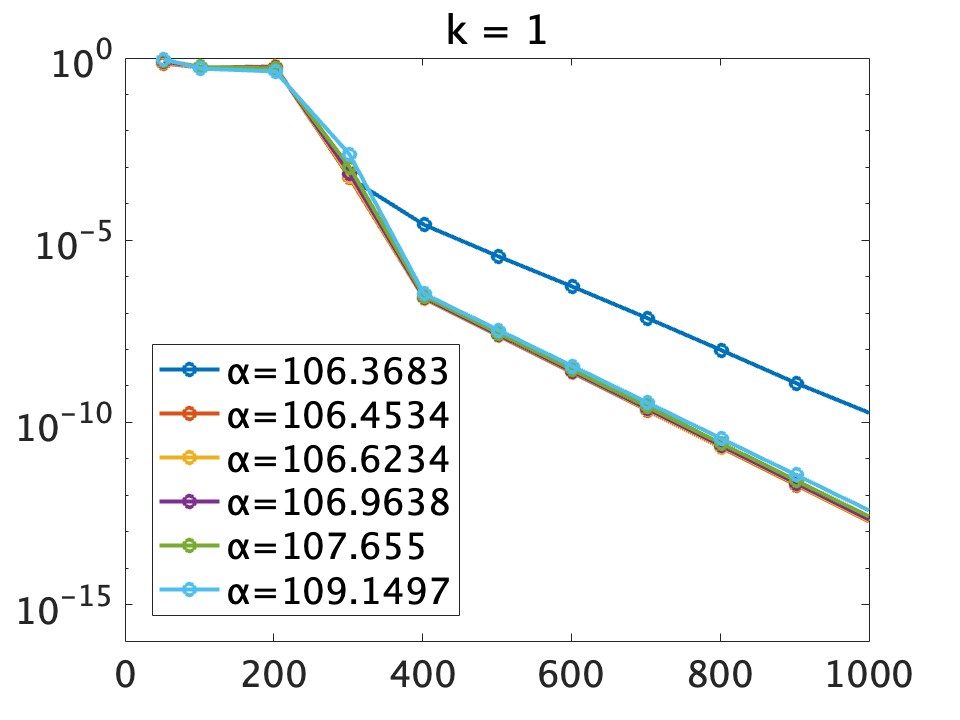}
  \end{minipage}
  \begin{minipage}[b]{0.49\linewidth}
    \centering
    \includegraphics[bb = 0 0 960 720, width=0.9\linewidth]{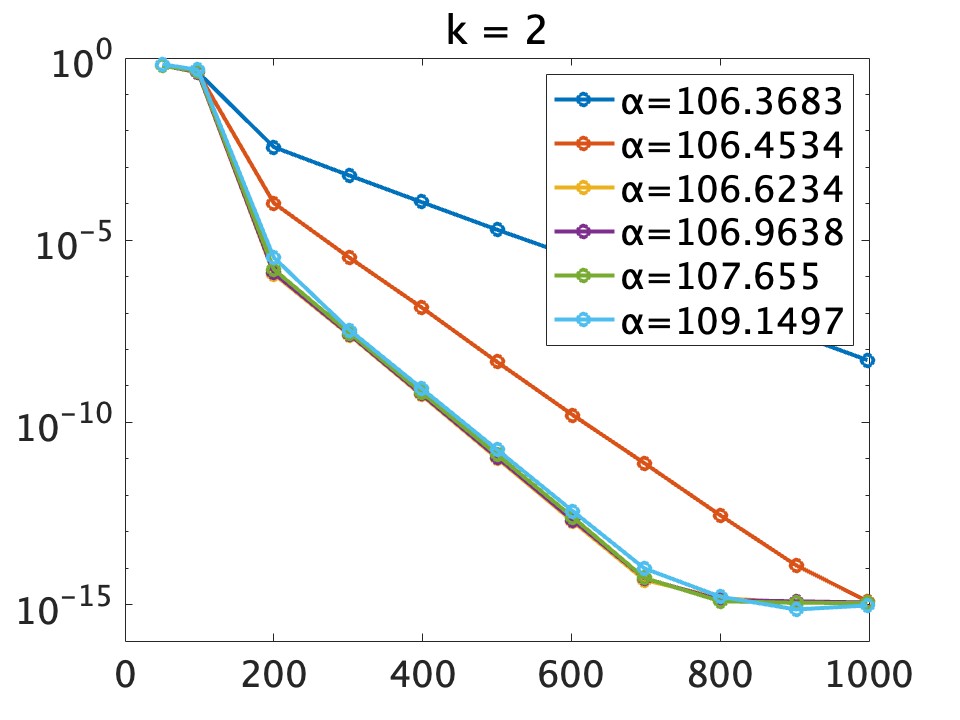}
  \end{minipage}
  
    \begin{minipage}[b]{0.49\linewidth}
    \centering
    \includegraphics[bb = 0 0 960 720, width=0.9\linewidth]{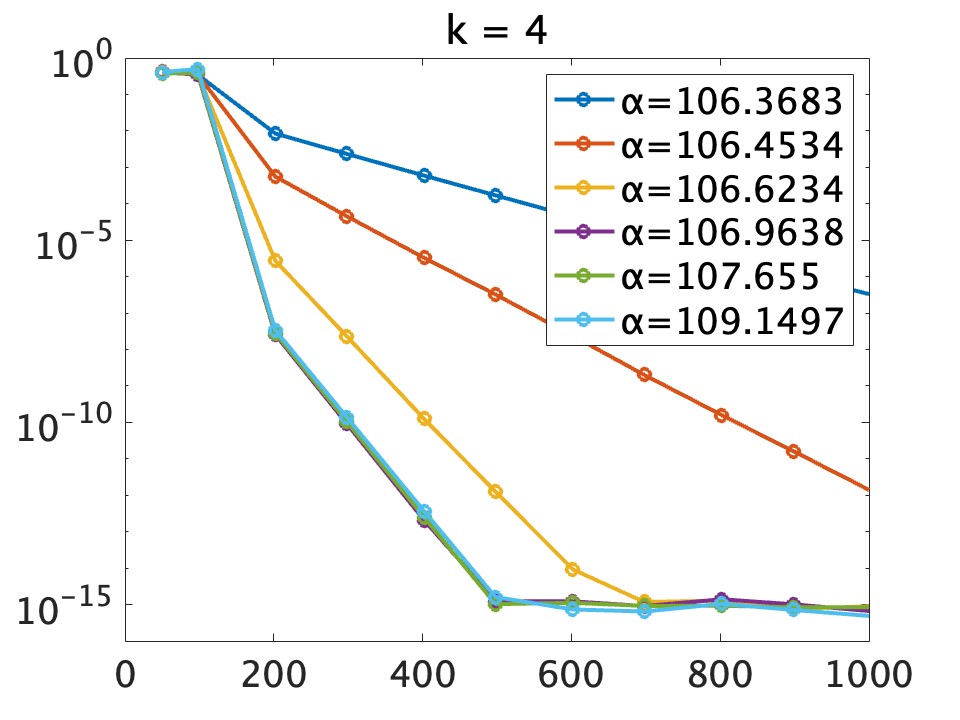}
  \end{minipage}
  \begin{minipage}[b]{0.49\linewidth}
    \centering
    \includegraphics[bb = 0 0 960 720, width=0.9\linewidth]{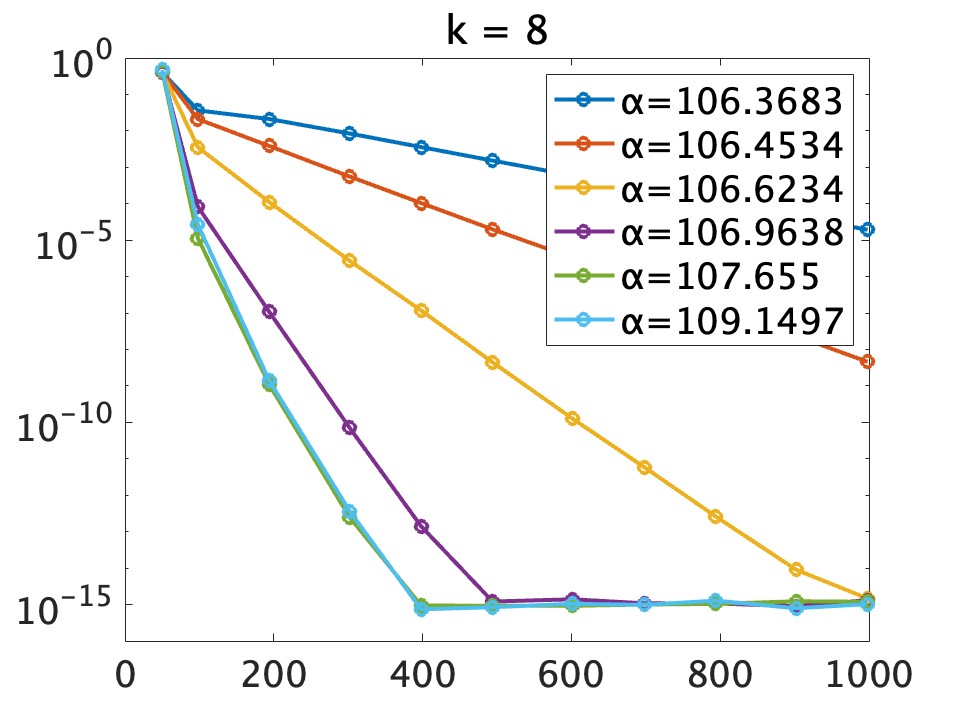}
  \end{minipage}
  
    \begin{minipage}[b]{0.49\linewidth}
    \centering
    \includegraphics[bb = 0 0 960 720, width=0.9\linewidth]{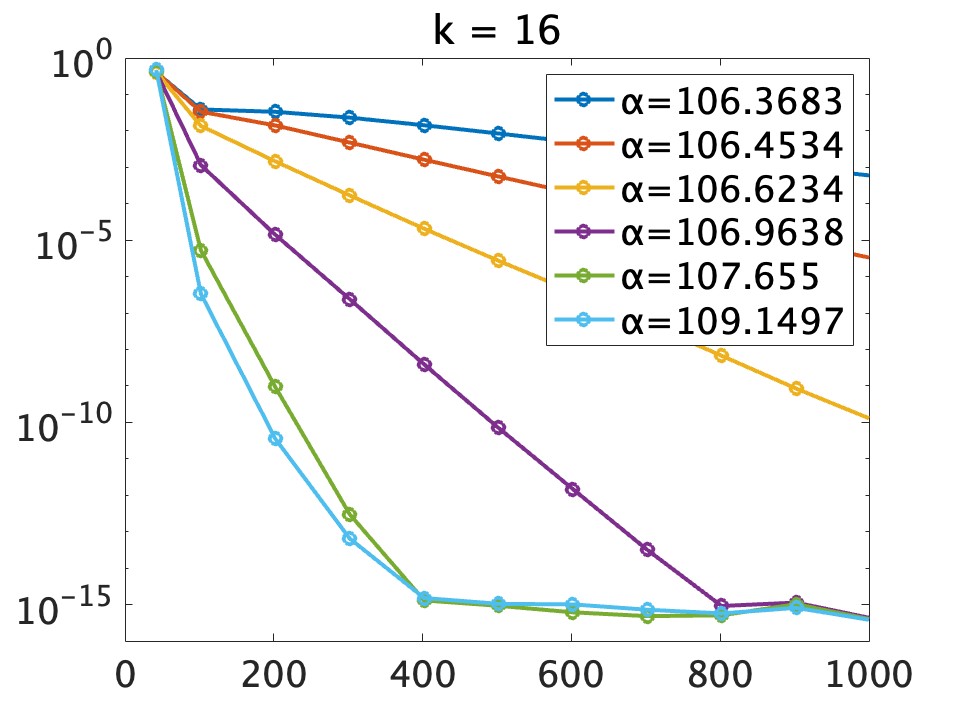}
  \end{minipage}
  \begin{minipage}[b]{0.49\linewidth}
    \centering
    \includegraphics[bb = 0 0 960 720, width=0.9\linewidth]{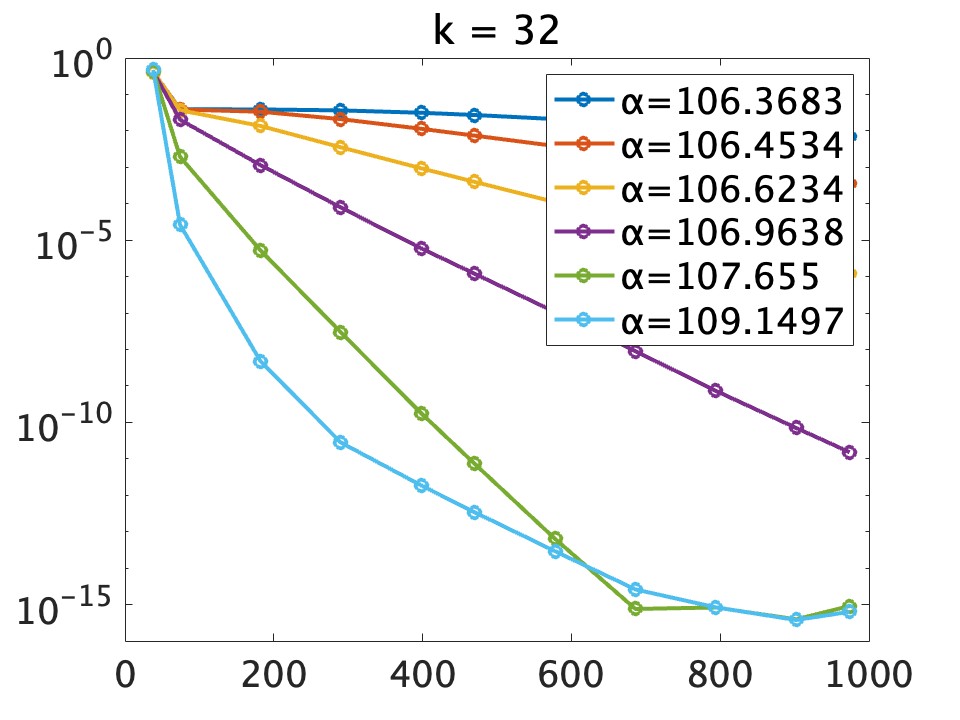}
  \end{minipage}
  
    \caption{The horizontal axis shows the total number of resolvent calculations (i.e., $4n+N+2=(4+k)n+2$), and the vertical axis shows the absolute error in the matrix 2-norm. The yellow line in the figure ``$k=4$'' is the parameter our method generates, and is used in Section \ref{sec:numerical_experiment_1}.}
  \label{fig:exp_alpha_k}
\end{figure}

\subsection{Comparison of the DE and Gauss-Laguerre Quadrature}
\label{sec:numerical_experiment_DE_GL}

In this section, we compare the double exponential formula with the Gauss-Laguerre quadrature for computing $I_{\alpha}(A)$. We used the test matrices $A_1,\, A_2$ and $A_3$. 

Figure \ref{fig:exp_3} shows the numerical results. 
As is seen in the figures, our DE method generally converges slower than the Gauss-Laguerre, 
especially when the imaginary parts of the eigenvalues are large.
Therefore, Gauss-Laguerre can also be a good choice for computing $I_\alpha(A)$. 
However, the parameter $d$, which is set in Theorem \ref{thm:total_error_of_I_alpha_n_h_for_A}, 
is conservative to derive the error bound, and thus, larger $d$ could be used in practical cases to improve convergence.
Such cases are shown in the bottom two figures of Figure \ref{fig:exp_3}.

We also note the following reasons why we propose the DE formula. 
The computation of the nodes and weights of the Gauss-Laguerre quadrature tends to be unstable when the number of nodes is large, while DE is very stable. The DE formula can be used for adaptive quadrature because it is based on equispaced nodes of the trapezoidal formula (i.e., we can reuse function values when doubling evaluation points).

\begin{figure}[htbp]
  \begin{minipage}[b]{0.49\linewidth}
    \centering
    \includegraphics[bb = 0 0 960 720, width=0.9\linewidth]{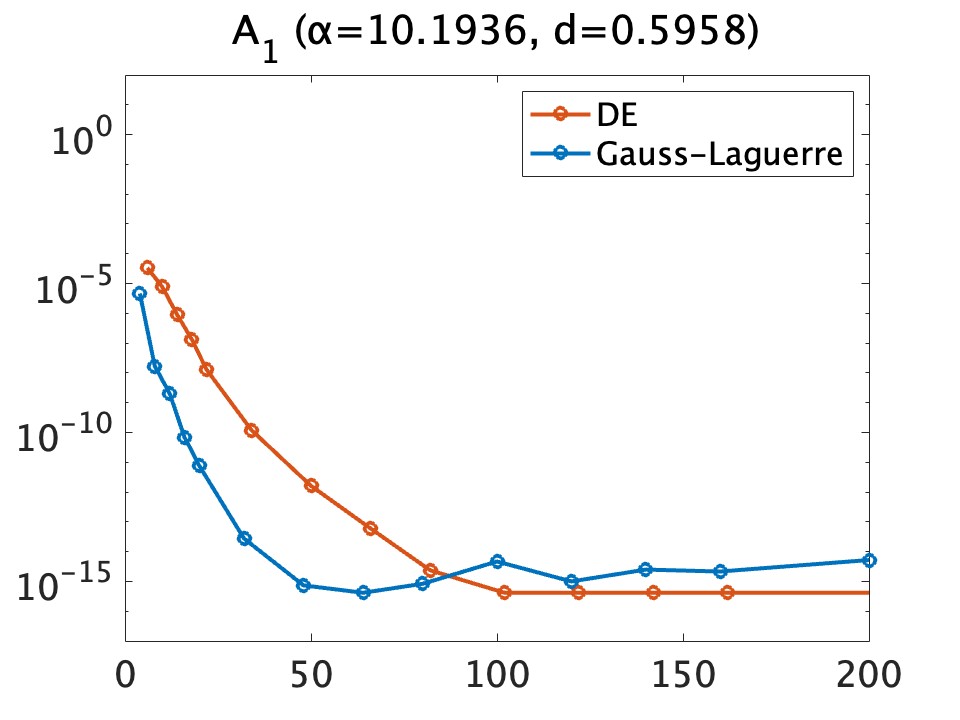}
  \end{minipage}
  \begin{minipage}[b]{0.49\linewidth}
    \centering
    \includegraphics[bb = 0 0 960 720, width=0.9\linewidth]{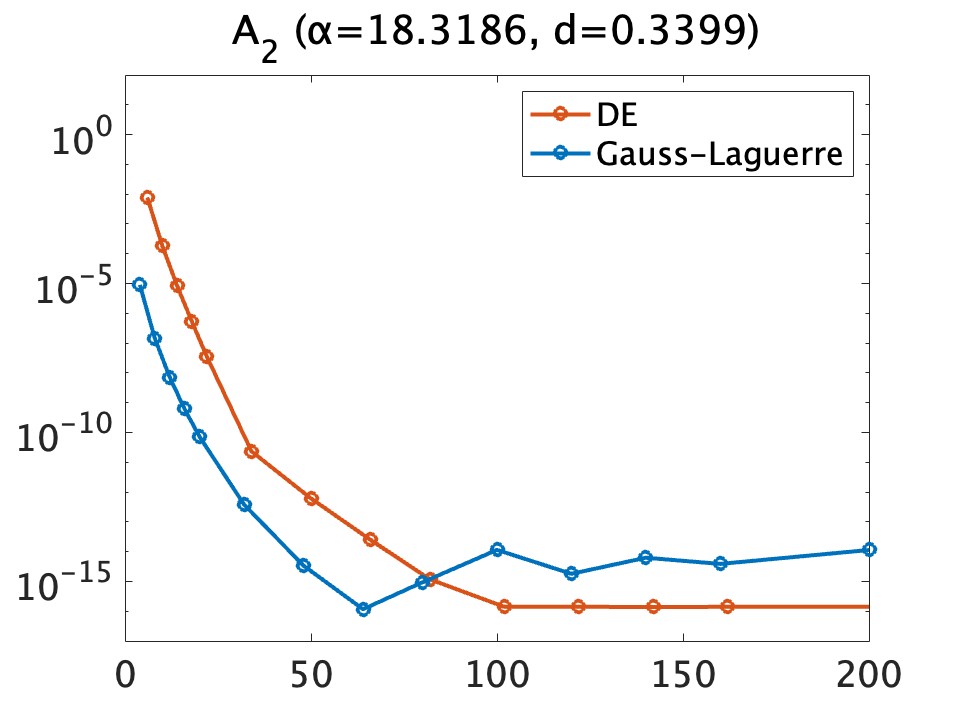}
\end{minipage}
  \begin{minipage}[b]{0.49\linewidth}
    \centering
    \includegraphics[bb = 0 0 960 720, width=0.9\linewidth]{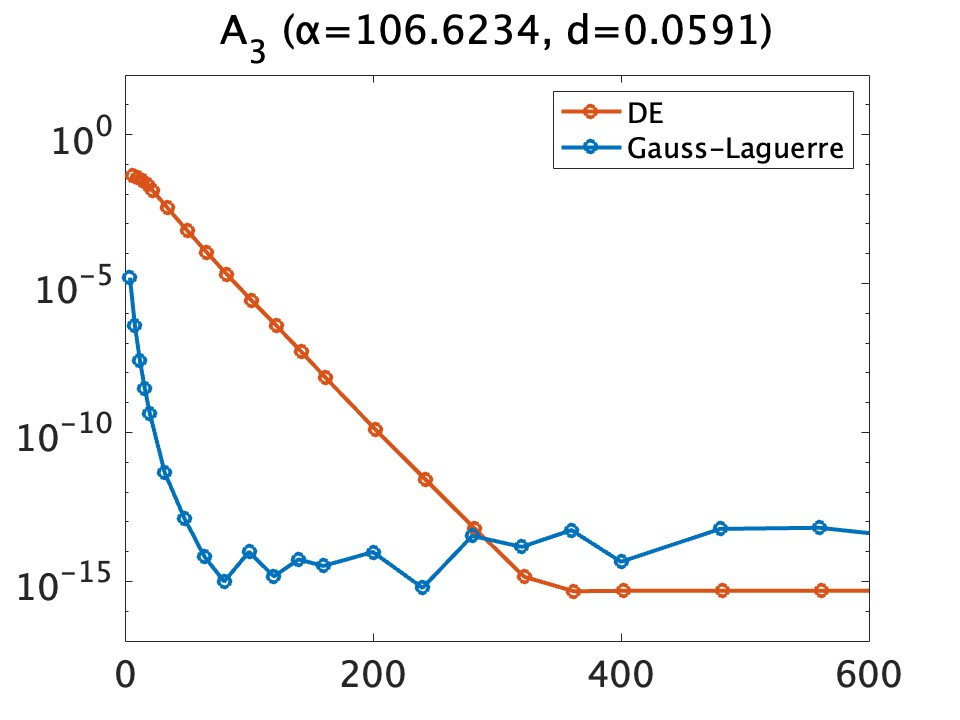}
  \end{minipage}
  \begin{minipage}[b]{0.49\linewidth}
    \centering
    \includegraphics[bb = 0 0 960 720, width=0.9\linewidth]{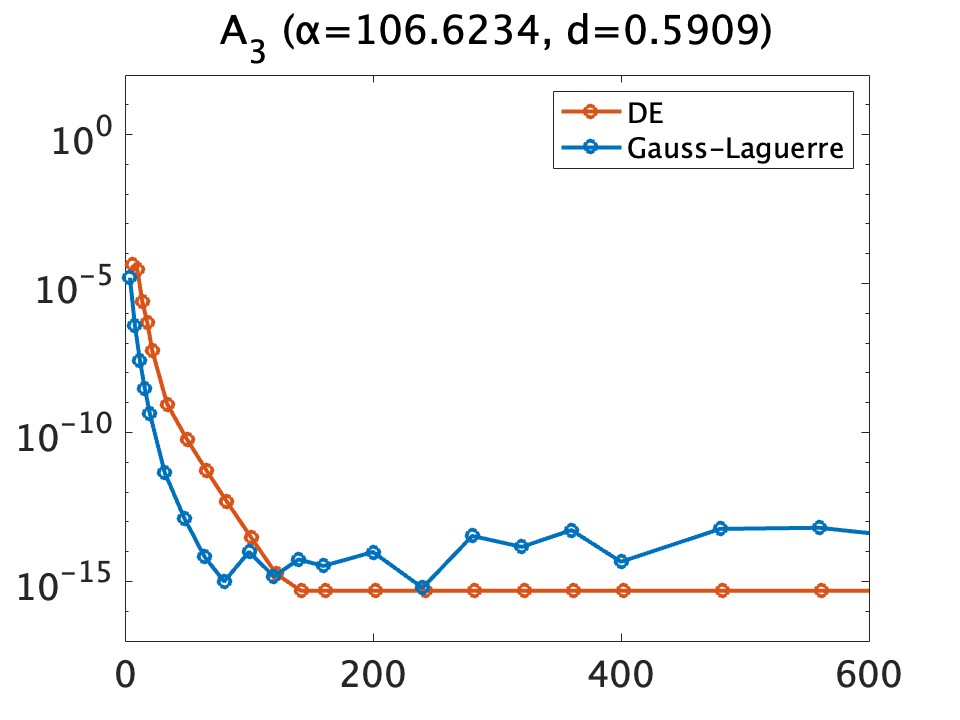}
\end{minipage}
\caption{The horizontal axis shows the total number of resolvent calculations, 
and the vertical axis shows the absolute error of $I_{\alpha}(A_i)$ in the matrix 2-norm. 
The three figures on the top left, top right, and bottom left show the case of 
$A_1,\, A_2,$ and $A_3$ respectively, 
where the parameters were chosen by our proposed method. 
In the bottom right experiment, we used a larger $d$ for $A_3$.
}
  \label{fig:exp_3}
\end{figure}

\section{Proofs}
\label{sec:proofs}

\subsection{Proof of Theorem~\ref{thm:total_error_of_I_alpha_n_h}}
\label{sec:proof_total_error_of_I_alpha_n_h}

For $d$ with $0 < d < \pi/2$, 
let $\mathcal{D}_{d}$ be defined by 
\begin{align}
\mathcal{D}_{d}
:=
\{ z \in \mathbb{C} \mid | \mathop{\mathrm{Im}} z | < d \}. 
\label{eq:def_strip_D_d}
\end{align}
To prove Theorem~\ref{thm:total_error_of_I_alpha_n_h}, 
we reduce it to the following theorem 
providing a bound of the error of the DE formula with $\phi$ in~\eqref{eq:def_DE_trans_4}. 

\begin{thm}[{\cite[Theorem 3.3]{tanaoka2023NMVTeng}},{\cite[Theorem 3.1]{tanaka2009function}}]
\label{thm:error_bound_for_DE4}
Let $d$, $K$, and $\beta$ be real numbers with $0 < d < \pi/2$, $K > 0$, and $0 < \beta \leq 1$, respectively. 
Let $f$ be an analytic function in $\phi(\mathcal{D}_{d})$ with
\begin{align}
|f(y)| \leq K \left| \left(\frac{y}{1+y} \right)^{\beta-1} \exp(-\beta y) \right|
\label{eq:f_bound_on_phi_D_d_for_DE4}
\end{align}
for any $y \in \phi(\mathcal{D}_{d})$. 
For an integer $n$ with $n > \beta/(4d)$, 
define $h$ by
\begin{align}
h := \frac{\log (4dn/\beta)}{n}. 
\label{eq:def_general_h_for_DE4}
\end{align}
Then, there exists a real number $C > 0$ independent of $n$ such that
\begin{align}
\left|
\int_{0}^{\infty} f(x) \, \mathrm{d}x 
-
h \sum_{k=-n}^{n} f(\phi(kh)) \phi'(kh)
\right|
\leq
C \exp \left( - \frac{2\pi d n}{\log(4dn/\beta)} \right).
\label{eq:error_bound_for_DE4}
\end{align}
\end{thm}

\begin{rem}
When we consider the map $\phi$ in~\eqref{eq:def_DE_trans_4} for complex variables, 
we need to specify a branch of the Riemann surface of the logarithm function defining $\phi$. 
Throughout this paper, 
we choose the branch
so that $(r, \theta) \mapsto \log (r \mathrm{e}^{\mathrm{i} \theta})$ becomes a continuous function 
on the region $\{ (r, \theta) \in \mathbb{R}^{2} \mid r \in (0, \infty), \ \theta \in \mathbb{R} \}$
with $\log (1 \cdot \mathrm{e}^{\mathrm{i} \cdot 0}) = 0$. 
\end{rem}

\begin{rem}
The variable transformation $t \mapsto \log (1 + \exp((\pi/2) \sinh t))$ is used in  \cite{tanaka2009function}. 
The constant $\pi/2$ in this transformation is different from that of $\phi$ in~\eqref{eq:def_DE_trans_4}, 
which is used in \cite{tanaoka2023NMVTeng}. 
However, 
this difference does not affect essential parts of the proof of Theorem~\ref{thm:error_bound_for_DE4}.
We adopt the settings of \cite{tanaoka2023NMVTeng} in this paper. 
\end{rem}

\begin{rem}
\label{rem:const_of_error_bound}
From the proof of Theorem~\ref{thm:error_bound_for_DE4} in \cite{tanaka2009function}, 
we know that $C$ in~\eqref{eq:error_bound_for_DE4} 
can be given by the product of $K$ in~\eqref{eq:f_bound_on_phi_D_d_for_DE4} 
and a positive real number $c_{d, \beta}$ 
depending only on $d$ and $\beta$. 
\end{rem}

To reduce Theorem~\ref{thm:total_error_of_I_alpha_n_h} to Theorem~\ref{thm:error_bound_for_DE4}, 
we begin with giving a simple region covering $\phi(\mathcal{D}_{d})$. 
To this end, 
for $t$ with $0 < t < \pi/2$, 
we define $s_{\ast}(t)$ by
\begin{align}
s_{\ast}(t) 
:=
\mathrm{arsinh} \left( \frac{\log 2}{\pi \cos t} \right). 
\label{eq:def_s_ast_t}
\end{align}
This is a monotone increasing function of $t$. 
In addition, 
for $x \in \mathbb{R}$ we define $a_{d}(x)$ and $b_{d}(x)$ by
\begin{align}
& a_{d}(x) 
:= 
(\tan d) (x + \log 2) + 2 \pi
\qquad \text{and}
\label{eq:def_a_d_t} \\
& b_{d}(x)
:=
a_{d}(\max \{ 0, x \}), 
\label{eq:def_b_d_t}
\end{align}
respectively. 
By using the function $b_{d}$, we define $\mathcal{H}_{d}$ by
\begin{align}
\mathcal{H}_{d}
:=
\{ \zeta \in \mathbb{C} \mid |\mathop{\mathrm{Im}} \zeta| \leq b_{d}(\mathop{\mathrm{Re}} \zeta) \}. 
\label{eq:region_between_cut_hyperbola}
\end{align}

\begin{rem}
\label{thm:a_d_zero_is_up_bd}
The value $a_{d}(0)$ is an upper bound of $\pi \cosh(s_{\ast}(d)) \sin d + \pi$. 
Indeed, this assertion is shown as follows: 
\begin{align}
\pi \cosh(s_{\ast}(d)) \sin d + \pi
& = \pi \sqrt{1 + \sinh^{2}(s_{\ast}(d))} \sin d + \pi 
\notag \\
& = \pi \sqrt{1 + \left( \frac{\log 2}{\pi \cos d} \right)^{2}} \sin d + \pi 
\notag \\
& = \sqrt{(\tan^{2} d) (\log 2)^{2}  + \pi^{2} \sin^{2} d} + \pi
\notag \\
& \leq 
(\tan d) \log 2  + \pi \sin d + \pi
\leq 
a_{d}(0), 
\label{eq:a_d_zero_is_up_bd}
\end{align}
where the second equality is owing to~\eqref{eq:def_s_ast_t}. 
\end{rem}

We can show that $\mathcal{H}_{d}$ covers $\phi(\mathcal{D}_{d})$. 

\begin{lem}
\label{thm:phi_D_d_in_H_d}
Let $d$ be a real number with $0 < d < \pi/2$. 
Then, we have
\begin{align}
\phi(\mathcal{D}_{d})
\subset
\mathcal{H}_{d}
\label{eq:inclusion_phi_D_d}
\end{align}
where $\phi$ and $\mathcal{H}_{d}$ are defined by~\eqref{eq:def_DE_trans_4} and~\eqref{eq:region_between_cut_hyperbola}, 
respectively. 
\end{lem}

\begin{proof}
Because it holds for $s ,t \in \mathbb{R}$ that
\begin{align}
1 + \exp(\pi \sinh(s + \mathrm{i} t))
= & \, 
(1 + \mathrm{e}^{\pi \sinh s \cos t} \cos(\pi \cosh s \sin t))
\notag \\
& \, + 
\mathrm{i}\,  
\mathrm{e}^{\pi \sinh s \cos t} \sin(\pi \cosh s \sin t), 
\label{eq:1_exp_complex_expand}
\end{align}
the region $\phi(\mathcal{D}_{d})$ is symmetric with respect to the real line. 
In addition, 
it is clear that $\phi(\mathbb{R}) \subset \mathcal{H}_{d}$. 
Therefore it suffices to show that 
\begin{align}
\phi(\{ s + \mathrm{i} t \in \mathbb{C} \mid s \in \mathbb{R} \})
\subset
\mathcal{H}_{d}
\label{eq:sufficient_for_phi_D_d}
\end{align}
for any $t$ with $0 < t < d$. 
We fix $t$ in this range and define $r(s,t)$ and $\theta(s,t)$ by
\begin{align}
& r(s,t) := \mathrm{e}^{\pi \sinh s \cos t} \qquad \text{and} 
\label{eq:def_r_s_t} \\
& \theta(s,t) := \pi \cosh s \sin t, 
\label{eq:def_theta_s_t}
\end{align}
respectively, 
and omit $(s,t)$ of these functions in the following. 
Then it follows from~\eqref{eq:1_exp_complex_expand} that
\begin{align}
\phi(s + \mathrm{i} t)
& =
\log((1+r \cos \theta) + \mathrm{i} r \sin \theta)
\notag \\
& = 
\frac{1}{2} \log(1 + 2r \cos \theta + r^{2}) + 
\mathrm{i} \, \mathrm{arg}(1+ r \mathrm{e}^{\mathrm{i} \, \theta}), 
\label{eq:phi_s_t_complex_expand}
\end{align}
where $\mathrm{arg}$ is considered in the Riemann surface of the logarithm function
and takes values in $(-\infty, \infty)$. 

We begin with estimating the imaginary part of $\phi(s + \mathrm{i} t)$ 
in~\eqref{eq:phi_s_t_complex_expand}. 
If $s \leq 0$, i.e., $r \leq 1$,  
the real part of  $(1+ r \mathrm{e}^{\mathrm{i} \, \theta})$ is non-negative
and 
$|\mathrm{arg}(1+ r \mathrm{e}^{\mathrm{i} \, \theta})| \leq \pi/2$ holds. 
If $s > 0$, i.e., $r > 1$, 
$\mathrm{arg}(1+ r \mathrm{e}^{\mathrm{i} \, \theta})$
is equal to the sum of $\theta$ and the angle between 
$r \mathrm{e}^{\mathrm{i} \, \theta}$ and 
$(1+ r \mathrm{e}^{\mathrm{i} \, \theta})$. 
Because the absolute value of the latter angle is less than $\pi$, 
we have
$|\mathrm{arg}(1+ r \mathrm{e}^{\mathrm{i} \, \theta})| \leq \theta + \pi$. 
Therefore it holds for any $s \in \mathbb{R}$ that
\begin{align}
| \mathop{\mathrm{Im}} \phi(s + \mathrm{i} t) |
\leq 
\begin{cases}
\pi/2 
& (s \leq 0), \\
\theta + \pi
& (s > 0).
\end{cases}
\label{eq:imag_phi_s_t}
\end{align}

Then 
we bound 
$| \mathop{\mathrm{Im}} \phi(s + \mathrm{i} t) |$ from above to show the relation in~\eqref{eq:sufficient_for_phi_D_d}. 
First we consider the case  $r \leq 2$. 
In this case, 
we have
\begin{align}
\mathrm{e}^{\pi \sinh s \cos t} \leq 2 
\iff 
\sinh s \leq \frac{\log 2}{\pi \cos t}
\iff
s \leq s_{\ast}(t)
\notag
\end{align}
and it follows from the last inequality, \eqref{eq:imag_phi_s_t}, and~\eqref{eq:def_theta_s_t} that
\begin{align}
| \mathop{\mathrm{Im}} \phi(s + \mathrm{i} t) |
\leq 
\begin{cases}
\pi/2 
& (s \leq 0), \\ 
\pi \cosh(s_{\ast}(t)) \sin t + \pi
& (0 < s \leq s_{\ast}(t)).
\end{cases}
\notag
\end{align}
The inequality in the latter case is owing to the monotonicity of $\cosh s$ for $s \geq 0$.
From these inequality and the monotonicity of $\cosh t \sin t$ for $t$ with $0 < t < d$,  
we finally have 
\begin{align}
| \mathop{\mathrm{Im}} \phi(s + \mathrm{i} t) |
\leq 
\pi \cosh(s_{\ast}(d)) \sin d + \pi
\leq a_{d}(0)
\label{eq:imag_phi_s_t_case_r_less_2}
\end{align}
for any $s \leq s_{\ast}(t)$. 
The last inequality is owing to~\eqref{eq:a_d_zero_is_up_bd}. 

Next we consider the case $r > 2$, i.e., $s > s_{\ast}(t)$. 
In this case,
it follows from~\eqref{eq:phi_s_t_complex_expand} that 
\begin{align}
\mathop{\mathrm{Re}} \phi(s + \mathrm{i} t)
& \geq 
\frac{1}{2} \log (1 - 2r + r^{2}) 
= 
\log(r - 1)
\geq
\log(r/2)
= 
\log r - \log 2, 
\label{eq:Re_phi_log_r}
\end{align}
which implies $\mathop{\mathrm{Re}} \phi(s + \mathrm{i} t) \geq 0$. 
Furthermore, 
by~\eqref{eq:def_r_s_t} and~\eqref{eq:def_theta_s_t}, 
$\log r$ is given by 
\begin{align}
\log r 
& = 
\pi \sinh s \cos t
= 
\pi \cos t \, (\cosh^{2} s - 1)^{1/2}
\notag \\
& = 
\pi \cos t \left(
\frac{\theta^{2}}{\pi^{2} \sin^{2} t} - 1
\right)^{1/2}
=
\frac{1}{\tan t} 
\left(
\theta^{2} - \pi^{2} \sin^{2} t
\right)^{1/2}. 
\label{eq:log_r_low_bd}
\end{align}
From~\eqref{eq:Re_phi_log_r} and~\eqref{eq:log_r_low_bd}, 
we get
\begin{align}
& (\tan^{2} t) (\mathop{\mathrm{Re}} \phi(s + \mathrm{i} t) + \log 2)^{2}
\geq 
\theta^{2} - \pi^{2} \sin^{2} t
\notag \\
& \iff
\theta^{2} 
\leq 
(\tan^{2} t) (\mathop{\mathrm{Re}} \phi(s + \mathrm{i} t) + \log 2)^{2} + \pi^{2} \sin^{2} t. 
\label{eq:theta_upper_bound}
\end{align}
Then, we can derive from~\eqref{eq:theta_upper_bound} and~\eqref{eq:imag_phi_s_t} a relation of 
$\mathop{\mathrm{Re}} \phi(s + \mathrm{i} t)$ and $\mathop{\mathrm{Im}} \phi(s + \mathrm{i} t)$ as
\begin{align}
|\mathop{\mathrm{Im}} \phi(s + \mathrm{i} t)|
& \leq
\sqrt{(\tan^{2} t) (\mathop{\mathrm{Re}} \phi(s + \mathrm{i} t) + \log 2)^{2} + \pi^{2}} + \pi
\notag \\
& \leq 
\sqrt{(\tan^{2} d) (\mathop{\mathrm{Re}} \phi(s + \mathrm{i} t) + \log 2)^{2} + \pi^{2}} + \pi
\notag \\
& \leq 
(\tan d) (\mathop{\mathrm{Re}} \phi(s + \mathrm{i} t) + \log 2) + \pi + \pi 
\notag \\
& = 
a_{d}(\mathop{\mathrm{Re}} \phi(s + \mathrm{i} t)). 
\label{eq:imag_phi_s_t_case_r_greater_2}
\end{align}

From Inequalities~\eqref{eq:imag_phi_s_t_case_r_less_2} and~\eqref{eq:imag_phi_s_t_case_r_greater_2}, 
we have 
\begin{align}
|\mathop{\mathrm{Im}} \phi(s + \mathrm{i} t)| 
\leq 
\max\{ a_{d}(0), a_{d}(\mathop{\mathrm{Re}} \phi(s + \mathrm{i} t)) \}. 
\notag
\end{align}
Because $a_{d}$ is monotone increasing, 
the right-hand side is equal to 
$a_{d}(\max\{0, \mathop{\mathrm{Re}} \phi(s + \mathrm{i} t) \}) = b_{d}(\mathop{\mathrm{Re}} \phi(s + \mathrm{i} t))$. 
Thus we get the relation in~\eqref{eq:sufficient_for_phi_D_d}. 
\end{proof}

By using Lemma~\ref{thm:phi_D_d_in_H_d}, 
we find a real number $d$ such that $f_{\alpha}(z, \cdot)$ 
defined by~\eqref{eq:def_integrand_of_non_osc} is analytic in $\phi(\mathcal{D}_{d})$. 

\begin{lem}
\label{thm:appropriate_d_for_D_d}
Let $z$ be a complex number with $\mathop{\mathrm{Re}} z < 0$ and 
let $\alpha$ be a real number with $|\mathop{\mathrm{Im}} z| + 2\pi < \alpha$. 
Then, the function $f_{\alpha}(z, \cdot)$ defined by~\eqref{eq:def_integrand_of_non_osc} 
is analytic in $\phi(\mathcal{D}_{d})$ for a real number $d$ satisfying~\eqref{eq:appropriate_d_for_D_d}, i.e.,
\begin{align}
0 < d < \mathrm{arctan} \left( \frac{\alpha - |\mathop{\mathrm{Im}} z| - 2 \pi}{- \mathop{\mathrm{Re}} z + \log 2} \right). 
\notag
\end{align}
\end{lem}

\begin{proof}
Let $u = \mathop{\mathrm{Re}} z$ and $v = \mathop{\mathrm{Im}} z$. 
As indicated by the domain of $f_{\alpha}(z, \cdot)$ in~\eqref{eq:domain_of_f_alpha}, 
the singularities of $f_{\alpha}(z, \cdot)$ are 
\[
 - u - \mathrm{i} (v \pm \alpha). 
\] 
By Lemma~\ref{thm:phi_D_d_in_H_d}, 
the function $f_{\alpha}(z, \cdot)$ is analytic in $\phi(\mathcal{D}_{d})$
if these singularities are not in $\mathcal{H}_{d}$. 
This condition is given by
\begin{align}
|v\pm \alpha| > b_{d}(- u), 
\notag
\end{align}
where $b_{d}$ is defined by~\eqref{eq:def_b_d_t}. 
Because $\alpha > |v|$ and $u< 0$, this condition is equivalent to 
$\alpha - |v| > a_{d}(-u)$,
where $a_{d}$ is defined by~\eqref{eq:def_a_d_t}. 
This condition is equivalent to,
\begin{align}
& \alpha - |v| > (\tan d) (-u + \log 2) + 2 \pi
\notag \\
& \iff 
\frac{\alpha - |v| - 2 \pi}{-u + \log 2} > \tan d. 
\label{eq:tan_d_condition}
\end{align}
This inequality implies the conclusion. 
\end{proof}

\begin{rem}
According to Lemma~\ref{thm:appropriate_d_for_D_d}, 
we need to choose $\alpha$ with $\alpha > 2\pi$ even if $\mathop{\mathrm{Im}} z = 0$.
This is owing to two overestimates in the proof of Lemma~\ref{thm:phi_D_d_in_H_d}
that are done for simplicity. 
First, 
we bound the angle between $r \mathrm{e}^{\mathrm{i} \theta}$ and $(1 + r \mathrm{e}^{\mathrm{i} \theta})$ by $\pi$. 
Second, 
we bound the term $\pi^{2} \sin^{2} d$ by $\pi^{2}$ in~\eqref{eq:imag_phi_s_t_case_r_greater_2}. 
However, 
thanks to these overestimates, 
we can get the simple formula for the range of $d$ as shown in~\eqref{eq:appropriate_d_for_D_d}. 
\end{rem}

We are in a position to prove Theorem~\ref{thm:total_error_of_I_alpha_n_h}. 
We apply Theorem~\ref{thm:error_bound_for_DE4} 
to the function $f_{\alpha}(z, \cdot)$ defined by~\eqref{eq:def_integrand_of_non_osc}. 

\begin{proof}[Proof of Theorem~\ref{thm:total_error_of_I_alpha_n_h}]
We bound 
\(
|f_{\alpha}(z, x)|
\)
for $x = \phi(s \pm \mathrm{i} t)$ with $s \in \mathbb{R}$ and $t \in (-d, d)$.  
To this end, we bound the absolute value of 
\begin{align}
\frac{1}{z\pm\mathrm{i}\alpha+x}= \frac{1}{x-(-z\mp \mathrm{i}\alpha)}\label{eq:part_to_frac}
\end{align}

Let $u = \mathop{\mathrm{Re}} z$ and $v = \mathop{\mathrm{Im}} z$. 
Then, 
by Lemmas~\ref{thm:phi_D_d_in_H_d}, ~\ref{thm:appropriate_d_for_D_d},
and the fact that $\mathcal{H}_d$ is symmetric with respect to the real axis, we have
\begin{align}
|x-(-z\mp \mathrm{i}\alpha))| 
& \geq 
\min \{ |x'-(-u- \mathrm{i}(v \pm \alpha))|  \mid x' \in \mathcal{H}_{d} \}
\notag \\
& \geq 
\min \{ |x' - (-u - \mathrm{i} (|v|\pm\alpha))| \mid x' \in \mathcal{H}_{d} \}. 
\label{eq:dist_between_sp_and_H_d}
\end{align}
By simple arguments of elementary geometry, 
we can show that the minimum value in the right-hand side of~\eqref{eq:dist_between_sp_and_H_d} is 
achieved by the distance between 
$(-u + \mathrm{i} (\alpha - |v|))$ and the half line 
$\{ p + \mathrm{i} \, a_{d}(p) \in \mathbb{C} \mid p \in [0, \infty) \}$, 
where $a_{d}$ is defined by~\eqref{eq:def_a_d_t}. 
See Figure~\ref{fig:distance} for this argument. 
Because the distance is equal to that between 
$(-u, \alpha - |v|)$ and $\{ (p, q) \in \mathbb{R}^{2} \mid (\tan d) p - q + (\tan d) \log 2 + 2 \pi = 0 \}$ in $\mathbb{R}^{2}$, 
it is given by
\begin{align}
& \frac{| - (\tan d) u - \alpha + |v| + (\tan d) \log 2 + 2 \pi|}{\sqrt{\tan^{2} d + 1}}
\notag \\
& = 
\frac{\alpha - |v| - 2 \pi - (\tan d) (-u + \log 2) }{1/\cos d}
\notag \\
& = 
(\alpha - |v| - 2 \pi) \cos d - (-u + \log 2) \sin d, 
\label{eq:formula_of_dist_between_sp_and_H_d}
\end{align}
which is positive because of~\eqref{eq:tan_d_condition}. 
Therefore we have
\begin{align}
|x-(-z\mp \mathrm{i}\alpha))| 
\geq
(\alpha - |v| - 2 \pi) \cos d - (-u + \log 2) \sin d. 
\label{eq:dist_y_and_ialpha_z}
\end{align}
Then, we deduce from~\eqref{eq:part_to_frac} that
\begin{align}
\left|\frac{1}{z\pm\mathrm{i}\alpha+x}\right|
\leq 
\frac{1}{(\alpha - |\mathop{\mathrm{Im}} z| - 2 \pi) \cos d - (-\mathop{\mathrm{Re}} z + \log 2) \sin d}. 
\label{eq:bound_of_frac_parts_of_f_alpha}
\end{align}

\begin{figure}
\centering
\includegraphics[bb = 0 0 256 255, width=0.5\linewidth]{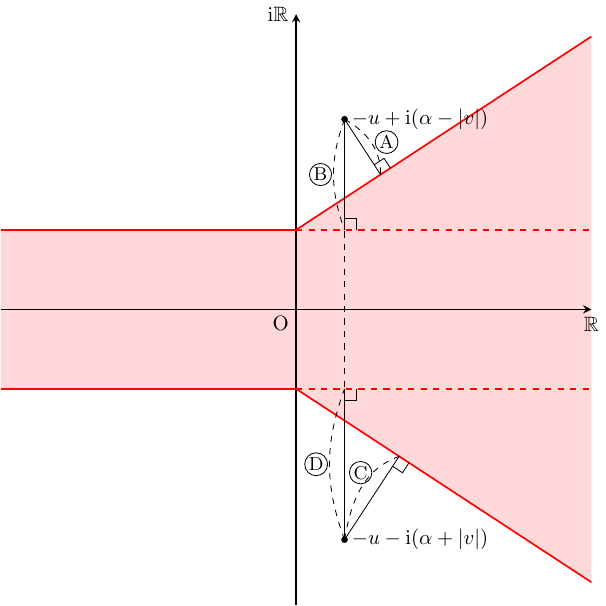}
\caption{
Explanation of the fact that the minimum value in the right-hand side of~\eqref{eq:dist_between_sp_and_H_d}
is given by~\eqref{eq:formula_of_dist_between_sp_and_H_d}.
The minimum value is the minimum distance between the two singular points 
(i.e., $-u + \mathrm{i} (\alpha - |v|)$ and $-u - \mathrm{i} (\alpha + |v|)$) 
and the boundary of the region $\mathcal{H}_{d}$. 
Candidates of the distance are A to D.
First, it clearly holds that A $\leq$ B, C $\leq $ D. 
Since $\mathcal{H}_d$ is symmetric w.r.t. the real axis, we have A $\leq$ C.
Therefore the minimum distance is given by A.  
}
\label{fig:distance}
\end{figure}

Then, 
by~\eqref{eq:bound_of_frac_parts_of_f_alpha} and 
the definition of $f_{\alpha}(z, \cdot)$ in~\eqref{eq:def_integrand_of_non_osc}, 
we have
\begin{align}
|f_{\alpha}(z, y)| 
\leq
K_{z, \alpha, d} \, |\exp(-y)|,
\label{eq:bound_of_f_alpha_z}
\end{align}
where $K_{z, \alpha, d}$ is defined by~\eqref{eq:def_of_K_alpha_z}, i.e., 
\begin{align}
K_{z, \alpha, d} 
= 
\frac{1/\pi}{(\alpha - |\mathop{\mathrm{Im}} z| - 2 \pi) \cos d - (-\mathop{\mathrm{Re}} z + \log 2) \sin d}.
\notag
\end{align}
It follows from~\eqref{eq:bound_of_f_alpha_z} and Lemma~\ref{thm:appropriate_d_for_D_d} that
the function $f_{\alpha}(z, \cdot)$ satisfies the assumption of Theorem~\ref{thm:error_bound_for_DE4}
with $K = K_{z, \alpha, d}$ and $\beta = 1$. 
Therefore
we have the error estimate in~\eqref{eq:total_error_of_I_alpha_n_h}
for $h = \log(4dn)/n$ in~\eqref{eq:def_of_h}
by taking Remark~\ref{rem:const_of_error_bound} into account. 
\end{proof}

\subsection{Proof of Theorem~\ref{thm:total_error_of_J_alpha_N}}
\label{sec:proof_total_error_of_J_alpha_N}

To prove Theorem~\ref{thm:total_error_of_J_alpha_N}, 
we reduce it to Theorem~\ref{thm:error_of_GL} below 
that provides a bound of the error of the Gauss-Legendre formula. 
To describe this theorem, 
for a real number $\rho$ with $\rho > 1$, 
we define $\mathcal{E}_{\rho}$ by 
\begin{align}
\mathcal{E}_{\rho}
:=
\left\{ 
\left.
z = \frac{1}{2} \left( u + \frac{1}{u} \right)
\in 
\mathbb{C} \, 
\right| \, 
u = \rho \mathrm{e}^{\mathrm{i} \theta}, \ 
0 \leq \theta < 2\pi
\right\}.
\label{eq:def_Bernstein_ell}
\end{align}
This is known as the Bernstein ellipse, 
which is an ellipse with foci $\pm 1$, 
semi-major axis $(\rho+\rho^{-1})/2$, and 
semi-minor axis  $(\rho-\rho^{-1})/2$. 
Let $\mathcal{B}_{\rho}$ be the inside of $\mathcal{E}_{\rho}$, 
i.e., 
the open region containing the origin with boundary $\mathcal{E}_{\rho}$. 
The region $\mathcal{B}_{\rho}$ contains the interval $[-1,1]$. 

\begin{thm}[{\cite[Ineq.~(18)]{rabinowitz1969rough}}, {\cite[Theorem 19.3]{trefethen2019approximation}}]
\label{thm:error_of_GL}
Let $g$ be an analytic function in $\mathcal{B}_{\rho}$ for $\rho > 1$
and 
there exists a real number $C > 0$ such that $|g(z)| \leq C$ for any $z \in \overline{\mathcal{B}_{\rho}}$. 
Then the $N$-point Gauss--Legendre quadrature formula with $N \geq 2$ satisfies
\begin{align}
\left| \int_{-1}^{1} g(t) \, \mathrm{d}t 
- 
\sum_{i=1}^{N} w_{i} \, g(t_{i})
\right|
\leq
\frac{64}{15} \frac{C \rho^{-2(N-1)}}{\rho^{2} - 1}, 
\label{eq:general_error_of_GL}
\end{align}
where $w_{i}$ and $t_{i}$ are the weights and nodes of the $N$-point Gauss--Legendre quadrature formula, 
respectively. 
\end{thm}

\begin{rem}
Rabinowitz \cite[Ineq.~(18)]{rabinowitz1969rough} presents the upper bound
\begin{align}
2 \left( 2 + \frac{2}{4N^{2} - 1} \right) \left(\max_{z \in \mathcal{E}_{\rho}}|g(z)| \right) \frac{\rho^{-2(N-1)}}{\rho^{2} - 1}
\notag
\end{align}
for the $N$-point formula. 
Because 
\begin{align}
2 \left( 2 + \frac{2}{4N^{2} - 1} \right) 
\leq 
2 \left( 2 + \frac{2}{4\cdot2^{2} - 1} \right) 
= 
\frac{64}{15}
\notag
\end{align}
holds for any $N \geq 2$, 
we can get \eqref{eq:general_error_of_GL}. 
\end{rem}

To apply Theorem~\ref{thm:error_of_GL}
to $g_{\alpha}(z, \cdot)$ defined by \eqref{eq:def_hat_g_alpha_z}, 
we find a real number $\rho$ such that this function satisfies the assumption of this theorem. 

\begin{lem}
\label{thm:rho_hat_g_analytic}
Let $z$ be a complex number with $\mathop{\mathrm{Re}} z < 0$ and 
let $\alpha$ be a real number with $|\mathop{\mathrm{Im}} z| < \alpha$. 
In addition, 
let $\delta$ be a real number with $0 < \delta < |\mathop{\mathrm{Re}} z|/\alpha$. 
Then the function $g_{\alpha}(z, \cdot)$ defined by \eqref{eq:def_hat_g_alpha_z} 
is analytic in $\mathcal{B}_{\rho_{z, \alpha, \delta}}$ with 
$\rho_{z, \alpha, \delta}$ defined by \eqref{eq:def_rho_z_alpha}, i.e., 
\begin{align}
\rho_{z, \alpha, \delta} 
= 
|\mathop{\mathrm{Re}} z|/\alpha - \delta + \sqrt{(|\mathop{\mathrm{Re}} z|/\alpha - \delta)^{2} + 1}.
\notag
\end{align}
\end{lem}

\begin{proof}
Because 
\begin{align}
g_{\alpha}(z, x)
= 
\frac{\alpha}{2\pi} 
\frac{\mathrm{e}^{\mathrm{i} \alpha x}}{\mathrm{i}\alpha x-z} ,
\label{eq:equiv_expr_of_hat_g}
\end{align}
the singularity point of $g_{\alpha}(z, \cdot)$ is
\begin{align}
\frac{z}{\mathrm{i}\alpha}, 
\end{align}
whose imaginary part is $-\mathop{\mathrm{Re}} z/\alpha$. 
Therefore 
$g_{\alpha}(z, \cdot)$
is analytic in $\mathcal{B}_{\rho}$
if $(\rho - \rho^{-1})/2 < |\mathop{\mathrm{Re}} z|/\alpha$
holds. 
Because $\rho_{z, \alpha, \delta}$ in~\eqref{eq:def_rho_z_alpha}
satisfies 
\begin{align}
\frac{\rho_{z, \alpha, \delta} - \rho_{z, \alpha, \delta}^{-1}}{2} = |\mathop{\mathrm{Re}} z|/\alpha - \delta,
\label{eq:eq_rho_z_alpha_delta}
\end{align} 
the conclusion holds. 
\end{proof}

Then we are in a position to prove Theorem~\ref{thm:total_error_of_J_alpha_N}. 

\begin{proof}[Proof of Theorem~\ref{thm:total_error_of_J_alpha_N}]
We bound $|g_{\alpha}(z, x)|$ for $x \in \overline{\mathcal{B}_{\rho_{z, \alpha, \delta}}}$ 
with $\rho_{z, \alpha, \delta}$ defined by~\eqref{eq:def_rho_z_alpha}.
It follows from the proof of Lemma~\ref{thm:rho_hat_g_analytic} that 
the absolute values of the denominators $(\mathrm{i}\alpha x-z)$ in~\eqref{eq:equiv_expr_of_hat_g} 
for $x \in \overline{\mathcal{B}_{\rho_{z, \alpha, \delta}}}$
are bounded from below as follows: 
\begin{align}
|\mathrm{i}\alpha x-z| 
= \alpha \left|x+\mathrm{i}z/\alpha\right|
\geq
\alpha\left(\frac{|\mathop{\mathrm{Re}} z|}{\alpha} -\frac{\rho_{z, \alpha, \delta} - \rho_{z, \alpha, \delta}^{-1}}{2}\right)
= 
\alpha\delta.
\label{eq:lb_of_denom_of_hat_g}
\end{align}
Furthermore, the absolute value of $\exp(\mathrm{i}\alpha x)$ in~\eqref{eq:equiv_expr_of_hat_g} 
for $x \in \overline{\mathcal{B}_{\rho_{z, \alpha, \delta}}}$
is bounded from above as follows: 
\begin{align}
|\exp(\mathrm{i}\alpha x)|
& \leq 
\exp(\alpha |\mathop{\mathrm{Im}}x|)
\leq
\exp\left( \alpha \cdot \frac{\rho_{z,\alpha,\delta} - \rho_{z,\alpha,\delta}^{-1}}{2} \right)
\notag \\
& = 
\exp( \alpha(|\mathop{\mathrm{Re}} z|/\alpha - \delta))
\leq 
\exp(|\mathop{\mathrm{Re}} z|). 
\label{eq:ub_of_sin_of_hat_g}
\end{align}
Then, it follows from 
\eqref{eq:equiv_expr_of_hat_g}, \eqref{eq:lb_of_denom_of_hat_g}, and~\eqref{eq:ub_of_sin_of_hat_g} 
that
\begin{align}
|g_{\alpha}(z, x)|
\leq
\frac{\exp(|\mathop{\mathrm{Re}} z|)}{2\pi \delta} 
\label{eq:ub_hat_g_on_B}
\end{align}
holds for any $x \in \overline{\mathcal{B}_{\rho_{z, \alpha, \delta}}}$. 

By Lemma~\ref{thm:rho_hat_g_analytic} and \eqref{eq:ub_hat_g_on_B}, 
the function $g_{\alpha}(z, \cdot)$ satisfies the assumption of Theorem~\ref{thm:error_of_GL}
with $\rho = \rho_{z, \alpha, \delta}$ in \eqref{eq:def_rho_z_alpha} 
and $C =\exp(|\mathop{\mathrm{Re}} z|)/(2\pi \delta)$. 
Then, by applying Theorem~\ref{thm:error_of_GL}, 
we get the conclusion. 
\end{proof}

\subsection{Proof of Theorem~\ref{thm:total_error_of_I_alpha_n_h_for_A}}
\label{sec:proof_I_alpha_n_h_for_A}
 
A counterpart of Theorem~\ref{thm:error_bound_for_DE4} holds 
for any matrix-valued function $f: \phi(\mathcal{D}_{d}) \to \mathbb{C}^{m\times m}$ 
with 
the absolute values $|\cdot|$ in the left-hand sides 
of~\eqref{eq:f_bound_on_phi_D_d_for_DE4} 
and~\eqref{eq:error_bound_for_DE4}
replaced by a matrix norm $\| \cdot \|$. 
Therefore 
it suffices to show that 
$f_{\alpha}(A, \cdot)$ is analytic in $\phi(\mathcal{D}_{d})$
and 
there exist $K > 0$ and $\beta \in (0,1]$ such that 
\begin{align}
\| f_{\alpha}(A, y) \| \leq K \left| \left(\frac{y}{1+y} \right)^{\beta-1} \exp(-\beta y) \right|
\label{eq:f_bound_on_phi_D_d_for_DE4_for_f_alpha_A}
\end{align}
holds for any $y \in \phi(\mathcal{D}_{d})$. 

Because of the expression of $f_{\alpha}(A, \cdot)$ in~\eqref{eq:f_alpha_z_eq_A}, 
its analyticity follows from that of $f_{\alpha}(\lambda_{i}, \cdot)$ for $i=1,\ldots, m$, 
which can be shown in the same manner as Lemma~\ref{thm:appropriate_d_for_D_d}
under the settings of $\alpha$ and $d$ 
in Theorem~\ref{thm:total_error_of_I_alpha_n_h_for_A}. 

To show \eqref{eq:f_bound_on_phi_D_d_for_DE4_for_f_alpha_A}, 
we need to estimate the norm of the inverse of the matrices 
that appear in~\eqref{eq:f_alpha_z_eq_A}. 
To this end, 
we use the following lemmas. 
We give proofs of Lemmas \ref{thm:T_inv_2_norm} and \ref{thm:T_inv_F_norm} 
in Appendix~\ref{sec:T_inv_norm} for completeness. 

\begin{lem}[{\cite[Chap.~1, \S 4, Ineq.~(4.12)]{kato1995perturbation}}]
\label{thm:T_inv_gen_norm}
Let $T \in \mathbb{C}^{m \times m}$ be a $m \times m$ non-singular matrix. 
Then, 
for any norm $\| \cdot \|$, 
it follows that there is a constant $\gamma$ such that 
\begin{align}
\left\| T^{-1} \right\| 
\leq
\gamma \frac{\| T \|^{m-1}}{| \det T |}. 
\label{eq:T_inv_gen_norm}
\end{align}
The constant $\gamma$ is independent of $T$, 
depending only on the norm $\| \cdot \|$. 
\end{lem}

\begin{lem}[{\cite[Lem.~1]{kato1960estimation}}]
\label{thm:T_inv_2_norm}
In Lemma~\ref{thm:T_inv_gen_norm}, 
we can set $\gamma = 1$ if $\| \cdot \|$ is the $2$-norm $\| \cdot \|_{2}$. 
\end{lem}

\begin{lem}
\label{thm:T_inv_F_norm}
In Lemma~\ref{thm:T_inv_gen_norm}, 
we can set $\gamma = \sqrt{m}$ if $\| \cdot \|$ is the Frobenius norm $\| \cdot \|_{\mathrm{F}}$. 
\end{lem}

By using these lemmas, 
we prove Theorem~\ref{thm:total_error_of_I_alpha_n_h_for_A}. 

\begin{proof}[Proof of Theorem~\ref{thm:total_error_of_I_alpha_n_h_for_A}]
As already stated, 
$f_{\alpha}(A, \cdot)$ is analytic in $\phi(\mathcal{D}_{d})$
under the settings of $\alpha$ and $d$. 
In the following, 
let $\| \cdot \|$ be the $2$-norm $\| \cdot \|_{2}$ or 
the Frobenius norm $\| \cdot \|_{\mathrm{F}}$ of matrices. 
We bound 
\(
\| f_{\alpha}(A, x) \|
\)
for $x = \phi(s \pm \mathrm{i} t)$ with $s \in \mathbb{R}$ and $t \in (-d, d)$.  
To this end, we bound the norm of 
\begin{align}
\left((x\pm \mathrm{i}\alpha)I+A\right)^{-1}.
\label{eq:part_to_frac_z_A}
\end{align} 
Recall that $\lambda_{1}, \ldots, \lambda_{m}$ are the eigenvalues of $A$. 
Under the settings of $\alpha$ and $d$, 
we can bound $|x \pm \mathrm{i} \alpha + \lambda_{i}|$ from below 
in the same manner as \eqref{eq:dist_y_and_ialpha_z} in the proof of Theorem~\ref{thm:total_error_of_I_alpha_n_h}. 
That is, we have
\begin{align}
|x\pm \mathrm{i}\alpha +\lambda_i| 
& \geq
(\alpha - |\mathop{\mathrm{Im}} \lambda_{i}| - 2 \pi) \cos d - (-\mathop{\mathrm{Re}} \lambda_{i} + \log 2) \sin d
\geq
\ell_{A, \alpha, d}
\label{eq:lb_by_L_A_alpha_d}
\end{align}
for any $i = 1,\ldots, m$, 
where
\begin{align}
\ell_{A, \alpha, d}
:=
\min_{1 \leq i \leq m} 
\Big(
(\alpha - |\mathop{\mathrm{Im}} \lambda_{i}| - 2 \pi) \cos d - (-\mathop{\mathrm{Re}} \lambda_{i} + \log 2) \sin d
\Big). 
\label{eq:def_L_A_alpha_d}
\end{align}
By Lemma~\ref{thm:T_inv_gen_norm} and~\eqref{eq:lb_by_L_A_alpha_d}, 
we get a bound of 
$\| \left((x\pm \mathrm{i}\alpha)I+A\right)^{-1} \|$
as follows:
\begin{align}
\left\| \left((x\pm \mathrm{i}\alpha)I+A\right)^{-1} \right\|
& \leq
\frac{\gamma}{\ell_{A, \alpha, d}}
\prod_{i=1}^{m-1} 
\left( 
\frac{\left\| (x\pm \mathrm{i}\alpha)I+A \right\|}{|x\pm \mathrm{i}\alpha +\lambda_i|}
\right)
\notag \\
& \leq
\frac{\gamma}{\ell_{A, \alpha, d}}
\left( 
\frac{(|x| + \alpha) \| I \| + \| A \| }{ \max \{ |x| - \alpha - \| A \|_{2}, \ell_{A, \alpha, d} \}}
\right)^{m-1}.
\label{eq:ub_norm_of_y_alpha_A}
\end{align}
Because the right-hand side tends to $\gamma \| I \|^{m-1} / \ell_{A, \alpha, d}$ as $|x| \to \infty$, 
it has a bounded supremum over $x$:
\begin{align}
L_{A, \alpha, d, \| \cdot \|}
:=
\sup_{\begin{subarray}{c} x = \phi(s \pm \mathrm{i} t) \\ s \in \mathbb{R}, \, t \in (-d,d) \end{subarray}}
\frac{\gamma}{\ell_{A, \alpha, d}}
\left( 
\frac{(|x| + \alpha) \| I \| + \| A \| }{ \max \{ |x| - \alpha - \| A \|_{2}, \ell_{A, \alpha, d} \}}
\right)^{m-1}
< \infty.
\label{eq:def_M_A_alpha_d_norm}
\end{align}
From \eqref{eq:ub_norm_of_y_alpha_A} and~\eqref{eq:def_M_A_alpha_d_norm}, the norm of \eqref{eq:part_to_frac_z_A} can be bounded by
\begin{align}
\left\|
\left((x\pm \mathrm{i}\alpha)I+A\right)^{-1}
\right\|
\leq 
L_{A, \alpha, d, \| \cdot \|}.
\label{eq:up_first_term_of_f_alpha_A}
\end{align}

Then, 
we can derive from \eqref{eq:f_alpha_z_eq_A} and~\eqref{eq:up_first_term_of_f_alpha_A} that 
\begin{align}
\| f_{\alpha}(A, x) \| 
\leq
\frac{L_{A, \alpha, d, \| \cdot \|} }{\pi} \, |\exp(-x)|.
\end{align}
Thus $f_{\alpha}(A, \cdot)$ satisfies 
\eqref{eq:f_bound_on_phi_D_d_for_DE4_for_f_alpha_A}
with $K = L_{A, \alpha, d, \| \cdot \|} / \pi$ and $\beta = 1$. 
Therefore
we have the error estimate in~\eqref{eq:total_error_of_I_alpha_n_h_for_A}
for $h = \log(4dn)/n$ in~\eqref{eq:def_of_h_for_A}
by the counterpart of Theorem~\ref{thm:error_bound_for_DE4}
mentioned at the beginning of Section~\ref{sec:proof_I_alpha_n_h_for_A}. 
\end{proof}

\subsection{Proof of Theorem~\ref{thm:total_error_of_J_alpha_N_for_A}}
\label{sec:proof_J_alpha_N_for_A}

\begin{proof}
A counterpart of Theorem~\ref{thm:error_of_GL} holds  
for any matrix-valued function $g: \overline{\mathcal{B}_{\rho}} \to \mathbb{C}^{m\times m}$ 
with the absolute values $|\cdot|$ in the statement replaced by a matrix norm $\| \cdot \|$. 
Therefore it suffices to show that 
$g_{\alpha}(A, \cdot)$
is analytic in $\mathcal{B}_{\rho_{A, \alpha, \delta}}$
and
there exists a real number $C > 0$ such that
$\| g_{\alpha}(A, x) \| \leq C$
holds for any $x \in \overline{\mathcal{B}_{\rho_{A, \alpha, \delta}}}$, 
where $\rho_{A, \alpha, \delta}$ is defined by~\eqref{eq:def_rho_z_alpha_for_A}. 

Because of the expression of $g_{\alpha}(A, \cdot)$ in~\eqref{eq:hat_g_alpha_z_eq_A}, 
its analyticity follows from that of $g_{\alpha}(\lambda_{i}, \cdot)$ for $i=1,\ldots, m$, 
which can be shown in the same manner as 
Lemma~\ref{thm:rho_hat_g_analytic}
under the settings of $\alpha$, $\delta$, and $\rho_{A, \alpha, \delta}$. 

To estimate the norm of
\begin{align}
g_{\alpha}(A, x)
=
\frac{\alpha}{2\pi} \exp(\mathrm{i}\alpha x) \left(\mathrm{i}\alpha x I-A \right)^{-1},
\label{eq:equiv_expr_of_hat_g_for_A}
\end{align}
we bound 
$\left\| \left(\mathrm{i}\alpha x I-A \right)^{-1} \right\|$ 
by using Lemma~\ref{thm:T_inv_gen_norm}. 
In the same manner as~\eqref{eq:lb_of_denom_of_hat_g}, 
we can bound 
$|\mathrm{i}\alpha x - \lambda_{i}|$ from below as 
\begin{align}
|\mathrm{i}\alpha x - \lambda_{i}|
\geq
\alpha\left(|\mathop{\mathrm{Re}} \lambda_{i}|/\alpha - \frac{\rho_{A, \alpha, \delta} - \rho_{A, \alpha, \delta}^{-1}}{2}\right)
\geq 
\alpha\delta
\notag
\end{align}
for $i = 1,\ldots, m$. 
Then, 
by Lemma~\ref{thm:T_inv_gen_norm}, 
we have 
\begin{align}
\left\| \left(\mathrm{i}\alpha x I-A \right)^{-1} \right\|
& \leq
\gamma \, \frac{\| \mathrm{i}\alpha xI - A \|^{m-1}}{\alpha^m \delta^{m}}
\notag \\
& \leq 
\frac{\gamma}{\alpha^m \delta^{m}} \left( \alpha |x| \| I \| + \| A \| \right)^{m-1}
\notag \\
& \leq 
\frac{\gamma}{\alpha \delta^{m}} 
\left(\left( \frac{\rho_{A, \alpha, \delta} + \rho_{A, \alpha, \delta}^{-1}}{2}\right) \| I \| + \frac{\| A \|}{\alpha} \right)^{m-1}
\label{eq:lb_of_denom_of_hat_g_for_A}
\end{align}
for any $x \in \overline{\mathcal{B}_{\rho_{A, \alpha, \delta}}}$. 
In addition, 
in a similar manner to~\eqref{eq:ub_of_sin_of_hat_g}, 
$|\exp(\mathrm{i}\alpha x)|$ is bounded as follows:
\begin{align}
|\exp(\mathrm{i}\alpha x)|
\leq 
\exp \left(\min_{1\leq i \leq m} |\mathop{\mathrm{Re}} \lambda_{i}| \right). 
\label{eq:up_sin_for_A}
\end{align}
Therefore it follows from
\eqref{eq:equiv_expr_of_hat_g_for_A}, 
\eqref{eq:lb_of_denom_of_hat_g_for_A}, and 
\eqref{eq:up_sin_for_A} 
that 
$\| g_{\alpha}(A, t) \| \leq C_{A, \alpha, \delta, \| \cdot \|}$, 
where
\begin{align}
C_{A, \alpha, \delta, \| \cdot \|}
:= 
\frac{\gamma}{2\pi \delta^{m}} 
\left( \left( \frac{\rho_{A, \alpha, \delta} + \rho_{A, \alpha, \delta}^{-1}}{2} \right) \| I \| + \frac{\| A \|}{\alpha}  \right)^{m-1}
\exp \left( \min_{1\leq i \leq m} |\mathop{\mathrm{Re}} \lambda_{i}| \right).
\label{eq:def_C_A_alpha_delta}
\end{align}

From the above arguments, 
we get the conclusion
by the counterpart of Theorem~\ref{thm:error_of_GL} 
mentioned at the beginning of this proof. 
\end{proof}

\section{Concluding remarks}
\label{sec:conclusion}

In this paper, 
we proposed a quadrature-based formula for computing the exponential function of matrices 
with a non-oscillatory integral on an infinite interval and an oscillatory integral on a finite interval. 
We applied the double-exponential (DE) formula and the Gauss-Legendre formula 
to the former and latter integrals, respectively. 
The error of the proposed formula is rigorously bounded as shown by 
Theorems~\ref{thm:total_error_of_I_alpha_n_h_for_A} and~\ref{thm:total_error_of_J_alpha_N_for_A}. 
From the results of the numerical experiments in Section~\ref{sec:num_exp}, 
we can observe that the error of the proposed formula is consistent with the theoretical analysis. 

There is room for improvement in the proposed formula and its theoretical analysis. 
As we pointed out in Sections~\ref{sec:numerical_experiment_alpha_k} and~\ref{sec:numerical_experiment_DE_GL}, 
the theoretical values of the parameters $\alpha$, $k$, and $d$ in the formula 
are not necessarily ideal. 
Therefore we should find a better way to determine these parameters from a practical point of view. 
In addition, 
finding better integral expressions for the exponential function and quadrature formulas for them remains a challenge.

\backmatter

%

\section*{Declarations}

\bmhead{Funding}

This work was supported by JSPS KAKENHI Grant Number 20H00581.

\bmhead{Competing interests}

The authors have no competing interests to declare that are relevant to the content of this article.


\begin{appendices}


\section{Proofs of Lemmas~\ref{thm:T_inv_2_norm} and~\ref{thm:T_inv_F_norm}}
\label{sec:T_inv_norm}

\begin{proof}[Proof of Lemma~\ref{thm:T_inv_2_norm}]
This proof is based on that of Lemma 1 in \cite{kato1960estimation}. 
Let $T = UH$ be the polar decomposition of $T$, 
where $U$ is a unitary matrix and $H$ is a positive-definite Hermitian matrix.  
Because the $2$-norm is invariant under unitary transformations, 
we have
\(
\| T \|_{2} = \| U H \|_{2} = \| H \|_{2}
\)
and
\(
\| T^{-1} \|_{2} = \| H^{-1} U^{-1} \|_{2} = \| H^{-1} \|_{2}
\).
Furthermore, 
\(
| \det T | = | \det U | | \det H | = | \det H |
\)
follows from the fundamental property of determinants. 
Therefore it suffices to show that 
\begin{align}
\left\| H^{-1} \right\|_{2} \leq \frac{\| H \|_{2}^{m-1}}{|\det H|}.
\label{eq:goal_ineq_for_H_2_norm}
\end{align}
Let $h_{1},\ldots, h_{m}$ be the eigenvalues of $H$ with $0 < h_{1} \leq \cdots \leq h_{m}$. 
Then, Inequality~\eqref{eq:goal_ineq_for_H_2_norm} is equivalent to 
\begin{align}
\frac{1}{h_{1}} 
\leq 
\frac{1}{h_{1}} \cdot \frac{h_{m}}{h_{2}} \cdots \frac{h_{m}}{h_{m}}, 
\notag
\end{align}
which holds obviously. Thus Inequality~\eqref{eq:goal_ineq_for_H_2_norm} holds. 
\end{proof}

\begin{proof}[Proof of Lemma~\ref{thm:T_inv_F_norm}]
Because the Frobenius norm is also invariant under unitary transformations, 
we can use the same procedure 
with the polar decomposition 
as the proof of Lemma~\ref{thm:T_inv_2_norm}. 
Therefore we have only to prove that
\begin{align}
\left\| H^{-1} \right\|_{\mathrm{F}} \leq \sqrt{m} \, \frac{\| H \|_{\mathrm{F}}^{m-1}}{|\det H|}
\label{eq:goal_ineq_for_H_F_norm}
\end{align}
holds for any positive-definite Hermitian matrix $H$. 
Let $h_{1},\ldots, h_{m}$ be the eigenvalues of $H$ with $0 < h_{1} \leq \cdots \leq h_{m}$. 
Then, 
Inequality~\eqref{eq:goal_ineq_for_H_F_norm} is equivalent to 
\begin{align}
\frac{1}{h_{1}^{2}} + \cdots + \frac{1}{h_{m}^{2}} 
\leq 
m \, \frac{(h_{1}^{2} + \cdots + h_{m}^{2})^{m-1}}{h_{1}^{2} \cdots h_{m}^{2}}. 
\notag
\end{align}
This inequality is shown as follows:
\begin{align}
m \, \frac{(h_{1}^{2} + \cdots + h_{m}^{2})^{m-1}}{h_{1}^{2} \cdots h_{m}^{2}}
& =
\sum_{i=1}^{m} 
\frac{(h_{1}^{2} + \cdots + h_{m}^{2})^{m-1}}{h_{1}^{2} \cdots h_{m}^{2}}
\notag \\
& = 
\sum_{i=1}^{m} 
\frac{1}{h_{i}^{2}} 
\prod_{\begin{subarray}{c} 1 \leq j \leq m \\  j \neq i \end{subarray}} \frac{h_{1}^{2} + \cdots + h_{m}^{2}}{h_{j}^{2}}
\geq 
\sum_{i=1}^{m} 
\frac{1}{h_{i}^{2}}. 
\notag 
\end{align}
Thus Inequality~\eqref{eq:goal_ineq_for_H_F_norm} holds. 
\end{proof}

\end{appendices}


\bibliography{bib_num_int_exp}

\end{document}